%% file: main.tex
\documentclass{article}

\input{preamble.tex}

\title{Canonical chain complexes for Morse-Smale vector fields}
\author{Clemens Bannwart\footnote{University of Modena and Reggio Emilia, Italy,
\texttt{clemens.bannwart@gmail.com}} \and Claudia Landi\footnote{University of Modena and Reggio Emilia, Italy, 
\texttt{claudia.landi@unimore.it}}}
\date{}

\begin{document}

\maketitle

\begin{abstract}
    In 1960, Smale defined a filtration of a closed smooth manifold by the unstable manifolds of fixed points and closed orbits of a Morse-Smale vector field defined on it, and derived generalized Morse inequalities. This suggests that, similarly to the Morse chain complex of a gradient-like vector field, even in the presence of closed orbits, Morse-Smale vector fields admit canonical chain complexes, invariant under topological equivalence, from which one can algebraically derive Morse inequalities. In this paper we show that this is actually the case, improving the state of the art that only offers non-canonical chain complexes.
    Technically, we achieve this result considering the \v{C}ech homology spectral sequence of the unstable manifolds filtration. In particular, we turn bounded exact couples into chain complexes such that the limit page of the spectral sequence associated with an exact couple gives the homology of the chain complex.
    We showcase our construction with examples.
\end{abstract}

\tableofcontents

\section*{Introduction}
\label{sec:intro}
\addcontentsline{toc}{section}{\nameref{sec:intro}}

When studying the topology of vector fields, i.e., those aspects that remain invariant under topological equivalence, algebraic invariants can be used to ignore many of the complexities while still capturing underlying structural behaviour. Thereby, a general theme  is to relate properties of vector fields with topological invariants of the manifold on which they are defined. 

An important family of relatively simple vector fields are the  Morse-Smale vector fields, which are well-behaved in terms of their fixed points, closed orbits, and connecting orbits, making them structurally stable. In two dimensions, structural stability is even equivalent to being Morse-Smale and a generic condition. Nevertheless, even for such vector fields, the program of assigning algebraic invariants to them is not yet fully understood, motivating us to address this gap.

We start from known results hinting at the existence of further structures and invariants that can be discovered. In the 1960s, Smale generalized the Morse inequalities from Morse functions to Morse-Smale vector fields \cite{Smale1960MorseIineq}. These inequalities relate the number of fixed points and closed orbits of each index with the Betti numbers of the underlying manifold. Moreover, considering the gradient-like case, i.e. those vector fields that have no closed orbits or otherwise recurrent behaviour, the Morse inequalities can  be derived algebraically from the existence of  a chain complex, variously ascribed in the literature to Morse, Smale, Thom, Witten, and Floer (here just called {\em Morse complex} for short).  The Morse complex, whose generators are given by the fixed points of the vector field, and whose  differential is defined by appropriately counting flow lines, has homology  isomorphic to the singular homology of the underlying manifold \cite{BanyagaLecturesOnMorseHomology}. 

Inspired by these results, the existence of a chain complex with properties similar to the Morse complex but applicable also in the presence of closed orbits has been variously investigated. One such complex has been proposed for Morse-Smale vector fields in \cite{EidiJost2022}, defining a chain complex in terms of the fixed points, closed orbits and connecting orbits. 
Some issues with this approach have been highlighted in \cite{AboutNonUniqueness2024}, which have later been addressed \cite{EidiJost2024arxiv}, although at the cost of losing the canonical nature of the chain complex, in the sense that some arbitrary choices may have to be made before applying the proposed construction.
Another popular approach, whose generality is even greater than Morse-Smale vector fields, is the connection matrix, first introduced by Franzosa \cite{FranzosaConnectionMatrix1989}. A problem with this is again the non-uniqueness in the presence of closed orbits, as explained in \cite{ReineckConnectionMatrix1990}.

Our contribution with this paper is a new method for assigning a chain complex to a Morse-Smale vector field on a closed smooth manifold that, unlike other chain complexes of Morse-Smale vector fields proposed in the literature, simultaneously enjoys the following properties: its homology gives the homology of the underlying manifold; it is invariant under topological equivalence of vector fields; and it is canonical in the sense that it does not require arbitrary choices. 

Technically, to construct this chain complex given a Morse-Smale vector field, we start with a filtration of the underlying manifold proposed by Smale in \cite{Smale1960MorseIineq}. In this filtration, the unstable manifolds of fixed points and closed orbits get added iteratively, depending on the existence of connecting orbits between each other. We then consider the \v{C}ech homology spectral sequence associated with this filtration, whose first page can be understood in terms of the fixed points and closed orbits of the vector field. We propose a method to turn any spectral sequence coming from a bounded exact couple into a chain complex. Applying this to the spectral sequence at hand yields our chain complex. Finally, we show how to derive the generalized Morse inequalities from this chain complex and explicitly compute the chain complex in two examples of vector fields on the 2-sphere.

\section{The setting}

\subsection{Stable and unstable manifolds}\label{sec:dyn-sys}

We introduce some concepts from dynamical systems that we are going to use, see e.g. \cite{palis2012geometric} for more details. Given a closed smooth manifold $M$, we denote by $\X(M)$ the set of smooth vector fields on $M$. By $\X^1(M)$ we denote the topological space with underlying set $\X(M)$ endowed with the Whitney $C^1$-topology.

In this subsection we fix a vector field $v\in \X(M)$ and consider the dynamics of the flow generated by $v$. We define everything with a subscript $v$, but this will be dropped when it is clear from the context which vector field we are considering.
We write $\phi_v\colon \R \times M \to M$ for the corresponding flow. 
The flow satisfies 
\[
\phi_v(s,\phi_v(t,p)) = \phi_v(s+t,p),
\qquad \quad
\phi_v(0,p)=p
\qquad \text{and} \qquad 
\left.\frac{d}{dt}\right\vert_{t=0} \phi_v(t,p) = v(p) 
\]
for all $p \in M$ and is uniquely determined by these properties. For a fixed number $t \in \R$, we sometimes write $\phi_v^t\colon M \to M$ for the diffeomorphism defined by $\phi_v^t(p) := \phi_v(t,p)$.

Given any point $p\in M$, the \textbf{orbit} of $p$ (w.r.t. $v$) is the set
\[
\Oo_v(p) := \phi_v(\R,p) = \{ \phi_v(t,p) \mid t \in \R \}.
\]
A \textbf{fixed point} (or \textbf{singular point}) of $v$ is a point $p\in M$ such that $\phi_v(t,p)=p$ for all $t\in \R$, i.e. $\Oo_v(p)=\{p\}$. This is equivalent to $v(p)=0$. A \textbf{periodic point} of $v$ is a point $p \in M$ which is not a fixed point such that there exists $T>0$ with $\phi_v(T,p)=p$. The smallest such $T$ is called the \textbf{period} of $p$. The orbit of a periodic point is called a \textbf{periodic orbit} (or \textbf{closed orbit}). One can check that all points on the same closed orbit have the same period.

\begin{definition}
    The \textbf{$\alpha$-limit set} and the \textbf{$\omega$-limit set} of $p$ (w.r.t. $v$) are defined as
    \begin{align*}
        \alpha_v(p) &:= \bigcap_{-\infty < t \le 0} \overline{\phi_v((-\infty,t],p)} = \{q \in M \mid \exists \text{ sequence } t_i \to -\infty \text{ with } \phi_v(t_i,p) \to q \}, \\
        \omega_v(p) &:= \bigcap_{0 \le t < \infty} \overline{\phi_v([t,\infty),p)} = \{q \in M \mid \exists \text{ sequence } t_i \to \infty \text{ with } \phi_v(t_i,p) \to q \}.
    \end{align*} 
\end{definition}

A subset $A\subseteq M$ is called \textbf{$\phi_v$-invariant} if $\phi_v(t,A)=A$ for all $t\in \R$. The orbit of any point $p\in M$ is a $\phi_v$-invariant subset of $M$.

\begin{proposition}[Proposition 1.4 of \cite{palis2012geometric}]
    Let $p\in M$. Then $\alpha_v(p)$ and $\omega_v(p)$ are non-empty, closed, connected and $\phi_v$-invariant subsets of $M$.
\end{proposition}

\begin{definition}
    A point $p\in M$ is called \textbf{chain-recurrent} (w.r.t. $v$) if for all $T>0$ and $\eps>0$ there exist $x_0,x_1,\ldots,x_m \in M$, with $x_0=x_m=p$, such that for all $i=0,\ldots,m-1$ there exists $t_i>T$ with $d_M(\phi(t_i,x_i),x_{i+1}) \le \eps$.
\end{definition}

\begin{remark}
    If $p\in M$ is either a fixed or periodic point, then $p$ is chain-recurrent.
\end{remark}

\begin{definition}
    Given a fixed point $p$ of $v$,  its \textbf{stable manifold} and its \textbf{unstable manifold} are defined, respectively, by
    \begin{align*}
        W^s_v(p) &:= \{q \in M \mid \phi_v(t,q) \to p \text{ as } t \to \infty \}, \\
        W^u_v(p) &:= \{q \in M \mid \phi_v(t,q) \to p \text{ as } t \to -\infty \}.
    \end{align*}
    Given a closed orbit $\gamma$ of $v$, we define its \textbf{stable manifold} and its \textbf{unstable manifold}, respectively,  by
    \begin{align*}
        W^s_v(\gamma) &:= \{q \in M \mid \exists p \in \gamma \text{ such that } d_M(\phi_v(t,p),\phi_v(t,q)) \to 0 \text{ as } t \to \infty \}, \\
        W^u_v(\gamma) &:= \{q \in M \mid \exists p \in \gamma \text{ such that } d_M(\phi_v(t,p),\phi_v(t,q)) \to 0 \text{ as } t \to -\infty \}.
    \end{align*} 
\end{definition}

\begin{definition}
    A set $\Lambda \subseteq M$ is called \textbf{hyperbolic} (w.r.t. $v$) if there exist two subbundles $E^s,E^u \subseteq T_\Lambda M$ such that 
    \begin{itemize}
        \item $T_pM = \langle v(p) \rangle \oplus E^s_p \oplus E^u_p$ for all $p \in \Lambda$,
        \item there exist $0< \lambda < 1$ and $C>0$ such that for all $p\in \Lambda$ and $t\ge 0$ we have
        \[
        |D\phi^t_v(p) x| \le C \lambda^t |x|
        \qquad \text{and} \qquad
        |D\phi^{-t}_v(p)y| \le C \lambda^t |y|
        \]
        where $x \in E^s_p$ and $y \in E^u_p$.
    \end{itemize}
    The two subbundles $E^s$ and $E^u$ are uniquely determined. The dimension of $E^u$ is called the \textbf{index} of $\Lambda$ and denoted by $\ind_v(\Lambda)$.
\end{definition}

In this paper we only consider the case where $\Lambda$ is either a fixed point or a closed orbit of $v$. The next two results were given by Smale in \cite{Smale_DiffDynSyst67} and the references therein. For a comprehensive treatment see \cite{FisherHassel2019HyperbolicFlows}.

\begin{theorem}
    The stable and unstable manifolds of hyperbolic fixed points and hyperbolic closed orbits are injectively immersed submanifolds of $M$.
\end{theorem}

More precisely, Smale proved the following result. 

\begin{lemma}[Lemma 3.8 in \cite{Smale1960MorseIineq}]\label{lem:unstable-manifold-homeomorphic}
    Let $v$ be a Morse-Smale vector field. If $\beta$ is a fixed point of $v$ of index $k\ge 0$, then $W^u(\beta) \cong \R^k$. If $\beta$ is a closed orbit of $v$ of index $k\ge 0$, then $W^u(\beta) \cong \R^k \times S^1$. 
\end{lemma}

\subsection{Morse-Smale vector fields}

\begin{definition}
    Two (injectively immersed) submanifolds $K,N \subseteq M$ are said to \textbf{intersect transversally}, if for every point $p\in K \cap N$ we have $T_pM = T_pK + T_p N$. The symbolic notation for this is $K \pitchfork N$.
\end{definition}

\begin{remark}
    If we have two submanifolds $K,N$ of $M$ such that $\dim(K)+\dim(N) < \dim(M)$, then $K$ and $N$ intersect transversally if and only if $K\cap N = \emptyset$. 
\end{remark}

\begin{definition}
    A vector field $v \in \X(M)$ is called a \textbf{Morse-Smale vector field}, if it satisfies the following conditions:
    \begin{enumerate}[(i)]
        \item The set of chain-recurrent points consists of finitely many fixed points and closed orbits, all of which are hyperbolic.
        \item The stable and unstable manifolds of fixed points and closed orbits intersect transversally.
    \end{enumerate}
    We call a Morse-Smale vector field \textbf{gradient-like} if it has no closed orbits.
\end{definition}

\begin{remark}
    Actually, hyperbolic fixed points and hyperbolic orbits are isolated, therefore on a closed manifold, by compactness we can have only finitely many. We could thus drop the condition that there are only finitely many fixed points and closed orbits. Also, some authors use non-wandering instead of chain-recurrent points in the definition of a Morse-Smale vector field. The two definitions are equivalent, as explained in \cite[Remark 5.3.47]{FisherHassel2019HyperbolicFlows}.
\end{remark}

The manifold decomposes into the unstable manifolds of fixed points and closed orbits. A proof of this result can be found in \cite{Smale_DiffDynSyst67} and references therein. For a more modern reference of this and some other relevant results about unstable manifolds, see \cite{DangRiv2020_SpectralAnalysisMSflows1}.

\begin{proposition}\label{prop:decomposition-unstable-manifolds}
    $M = \bigsqcup_\beta W^u(\beta)$, where the $\beta$ range over all fixed points and closed orbits.
\end{proposition}

\subsection{The unstable manifolds filtration of a Morse-Smale vector field}\label{sec:filtration-induced-by-vector-field}

We recall a way to filter the underlying manifold of a Morse-Smale vector field with respect to unstable manifolds introduced by Smale in \cite{Smale1960MorseIineq}. In the gradient-like case, one can consider an easier filtration, adding the unstable manifolds according to their dimension, i.e. adding the $k$-dimensional ones at time $k$. This is also called the {\em index filtration} and is equivalent to considering the $k$-skeleta of the corresponding CW decomposition of the manifold. In the presence of periodic orbits, one needs to be more careful, since the stable and unstable manifolds of periodic orbits of the same index may intersect. This means that if we simply filter according to dimensions, we do not get a CW filtration. The idea of Smale's filtration, later called the {\em unstable manifolds filtration}, is thus to define it recursively, adding in each step all unstable manifolds that attach to what has already been added before. Another approach that is often seen in the literature is to pick an {\em energy function} for the vector field as defined in \cite{MeyerEnergyFunctions68} and consider the sublevelset filtration with respect to that function.  However, such energy functions are not determined uniquely by the vector field and can lead to different filtrations. Choosing self-indexing functions in the presence of closed orbits can cause the same problems as with the index filtration.
The unstable manifolds filtration is a particular case of an attractor filtration as defined in \cite{Moeckel1988MorseDecompConnectionMat}, which is a filtration in a much broader sense.

We fix a field $\Fi$. All homology groups and vector spaces are understood to have coefficients in $\Fi$, which we usually do not make explicit in our notation.
Let $M$ be a closed smooth manifold of dimension $m$. Given a Morse-Smale vector field $v$ on $M$, we can decompose $M$ as a disjoint union of the unstable manifolds of the fixed points and closed orbits of $v$ by \Cref{prop:decomposition-unstable-manifolds}.

To make this more precise, we first define the boundary of an unstable manifold. If $\beta$ is either a fixed point or a closed orbit of index $k$, we denote by $\overline{W^u(\beta)}$ the closure of the unstable manifold of $\beta$ in $M$. We define $\partial W^u(\beta) := \overline{W^u(\beta)} \setminus W^u(\beta)$. Note that this is equivalent to the definition given in \cite{Smale1960MorseIineq}.
Now let $L_{-1}:=\emptyset$ and, for $i\ge 1$, let $L_i$ be the union of all those $W^u(\beta)$, where $\beta$ is a singular element of $v$, such that $\partial W^u(\beta) \subseteq L_{i-1}$. 

\begin{proposition}\label{prop:exist-number}
    There exists a number $n$ such that $L_n = M$.
\end{proposition}

\begin{proof}
    By \Cref{prop:decomposition-unstable-manifolds}, it is enough to show that every unstable manifold eventually gets added in the filtration. Note that the filtration can also be expressed in terms of a certain partial order $\le$ on the set of unstable manifolds of critical elements, where $W^u(\beta) \le W^u(\beta')$ if $W^u(\beta) \subseteq \overline{W^u(\beta')}$. The fact that this is indeed a partial order was proved in (either \cite{Smale_DiffDynSyst67} or \cite{Smale1960MorseIineq}), see also \cite[Theorem 3.7]{DangRiv2020_SpectralAnalysisMSflows1} for a more modern exposition. Every unstable manifold gets added precisely 1 step after the last unstable manifold below it with respect to $\le$ has been added. In other words, $W^u(\beta) \subseteq L_p \setminus L_{p-1}$, where $p$ is the height of $W^u(\beta)$ in the poset. Since we have finitely many critical elements, the height of this poset is finite, and hence, if we choose $n$ to be equal to the height of the poset, then $L_n = M$.
\end{proof}

Later, we will consider a spectral sequence that can be assigned to this filtration.

\section{Turning spectral sequences into chain complexes}\label{sec:turning-ss-into-ch}

This section is purely algebraic. We describe a method for turning a spectral sequence coming from an exact couple into a chain complex. The idea is to decompose the first page of the spectral sequence in such a way that the differentials from the different pages of the spectral sequence induce isomorphisms between certain summands of the decomposition. Taking the homology of this complex then gives us back the infinity page of the spectral sequence. We describe this procedure in the case where the spectral sequence is induced from an exact couple, since this condition holds true in the case we care about and it simplifies the notation.

\subsection{Preliminaries on spectral sequences and exact couples}\label{sec:spec-seq}

We state some definitions and results from the theory of spectral sequences. We adapt the definitions to our needs, since we are only working with field coefficients. More details (and more generality) can be found in \cite{weibel1994introduction}. Consider also that there are some errors with the indices in \cite[Section 5.9]{weibel1994introduction}, see \cite{WeibelHomErrata} for a list of corrections. 
We confine ourselves to spectral sequences and exact couples of {\em finite total dimension}. This means that all their vector spaces are finite dimensional and only finitely many of them are non-zero. This assumption is assumed implicitly throughout this paper.

\begin{definition}
    A \textbf{spectral sequence} starting at page $a \in \Z$ consists of the following data:
    \begin{itemize}
        \item vector spaces $E^r_{p,q}$ for all $r\ge a$ and $p,q \in \Z$,
        \item linear maps $d^r_{p,q}\colon E^r_{p,q} \to E^r_{p-r,q+r-1}$ for all $p,q$, satisfying $d^r_{p-r,q+r-1} \circ d^r_{p,q} = 0$,
        \item isomorphisms between $E^{r+1}$ and the homology of $E^r$, i.e.
            \[
            E^{r+1}_{p,q} \cong \ker(d^r_{p,q}) / \im(d^r_{p+r,q-r+1}).
            \]
    \end{itemize}
\end{definition}

The collection of the vector spaces $E^r$ together with the linear maps $d^r$ for some fixed value of $r$ is called the \textbf{$r$-th page} of the spectral sequence. The \textbf{total degree} of a term $E^r_{p,q}$ in a spectral sequence is the number $p+q$. A spectral sequence is called \textbf{bounded} if for each $n$ and $r$, there are only finitely many non-zero terms of total degree $n$ in the $r$-th page, and \textbf{bounded below} if for every $n$ there exists an integer $f(n)$ such that for $p<f(p+q)$ we have $E_{p,q}=0$. Note that this implies that for every $p,q$ there exists $r_0$ such that for all $r\ge r_0$ we have $E^r_{p,q}=E^{r+1}_{p,q}$. We denote this vector space by $E^\infty_{p,q}$. The collection of $E^\infty_{p,q}$ for all $p,q$ is called the \textbf{$\infty$-page} of the spectral sequence.

\begin{definition}
    A bounded spectral sequence is said to \textbf{converge to $H_*$}, where $H_*=(H_n)_{n \in \Z}$ is a collection of vector spaces, if for each $n$ there exists a finite filtration
    \[
    0=H^s_n \subseteq \cdots \subseteq H^{p-1}_n \subseteq H^p_n \subseteq H^{p+1}_n \subseteq \cdots \subseteq H^t_n =H_n
    \]
    and isomorphisms $E^\infty_{p,q} \cong H^p_n / H^{p-1}_n$, where $n=p+q$. In symbolic notation, we write $E \implies H_*$.
\end{definition}

Thanks to the fact that we are using field coefficients, convergence of a spectral sequence to a collection of vector spaces is enough to reconstruct these vector spaces from the infinity page of the spectral sequence up to isomorphism. For self-containment, we give a simple proof of this statement in the case of finite total dimension, which is sufficient for our purposes.

\begin{proposition}\label{prop:ss-converges-implies-isomorphism}
    If $E$ is a spectral sequence that converges to a collection of vector spaces $H_* = (H_n)_{n\in \Z}$, then $H_n \cong \bigoplus_{p+q=n}E^\infty_{p,q}$ for all $n \in \Z$.
\end{proposition}

\begin{proof}
    Since we are dealing with finite dimensional vector spaces, it suffices to show that the dimensions agree. We have
    \[
    \dim (H_n) = \sum_p \dim \left(H_n^p/H_n^{p-1}\right) = \sum_p \dim \left(E^\infty_{p,n-p}\right) = \dim \left( \bigoplus_p E^\infty_{p,n-p}\right) = \dim \left( \bigoplus_{p+q=n} E^\infty_{p,q}\right),
    \]
    where the first and third equality follow from standard linear algebra arguments, the second holds because $E$ converges to $H_*$, and in the fourth equality we are just reindexing.
\end{proof}

We will be interested in spectral sequences that come from exact couples.

\begin{definition}
    An \textbf{exact couple} (of order $r$, for some $r \in \Z$) 
    \begin{center}
        \begin{tikzcd}[row sep=25pt]
            \E\colon &D \ar[rr,"i"] &&D \ar[dl,"j"] \\
            &&E \ar[ul,"k"]
        \end{tikzcd}
    \end{center}
    consists of bigraded vector spaces $D=\bigoplus_{p,q} D_{p,q}$ and $E=\bigoplus_{p,q} E_{p,q}$ together with bigraded linear maps $i,j,k$ with bidegrees $(1,-1),(-r,r), (-1,0)$, satisfying the exactness conditions $\ker(i)=\im(k)$, $\ker(j)=\im(i)$, $\ker(k)=\im(j)$.
\end{definition}

If we have such an exact couple, the map $d:=jk$ is a differential (i.e. $d^2=0$) on $E$ and we may consider the resulting homology $H(E) := \ker(d)/\im(d)$. The bigrading on $E$ induces a bigrading on $H(E)$ (which we continue to hide from the notation). By replacing $E$ with $H(E)$ and $D$ with $i(D)$ we get another exact couple, as described in the next definition.

\begin{definition}
    Given an exact couple (of order $r$)
    \begin{center}
        \begin{tikzcd}[row sep=25pt]
            \E\colon &D \ar[rr,"i"] &&D \ar[dl,"j"] \\
            &&E \ar[ul,"k"]
        \end{tikzcd}
    \end{center}
    its \textbf{derived couple} is the exact couple (of order $r+1$)
    \begin{center}
        \begin{tikzcd}[row sep=25pt]
            \E'\colon &D' \ar[rr,"i'"] &&D' \ar[dl,"j'"] \\
            &&E' \ar[ul,"k'"]
        \end{tikzcd}
    \end{center}
    where $D' = i(D)$, $E'=H(E)$, $i'$ is the restriction of $i$, $k'([\alpha]) = k(\alpha)$, and $j'(i(x)) = [j(x)]$. The bidegrees of $i',j',k'$ are $(1,-1),(-(r+1),r+1), (-1,0)$, respectively. We will usually write $i(D)$ instead of $D'$, write $i$ for $i'$ as it is just the restriction, and write $k$ for $k'$ as it is just the induced map on the quotient. 
\end{definition}

A direct computation/diagram chase shows that the maps in the definition of the derived couple are well-defined and satisfy the exactness conditions, so that the derived couple is again an exact couple. Higher derived couples could be defined iteratively (i.e. the derived couple of the derived couple etc.), but it is also possible to give a direct definition. For this, we first define the subspaces 
\begin{align*}
    B^{(r)} &:= j(\ker(i^r)),
    &Z^{(r)} &:= k^{-1}(\im(i^r)).
\end{align*}
We use the convention that $i^0$ denotes the identity map on $D$, so $\ker(i^0)=\{0\}$ and $\im(i^0)=D$, thus $B^{(0)}=\{0\}$ and $Z^{(0)} = E$. Since the kernels of $i^r$ are growing and the images of $i^r$ are shrinking when $r$ gets larger (they can also stay the same), we have that $B^{(r)} \subseteq B^{(r+1)}$ and $Z^{(r+1)} \subseteq Z^{(r)}$. Moreover, all the $B^{(r)}$ are contained inside $j(D) = \im(j)$ and all the $Z^{(r)}$ contain $k^{-1}(\{0\}) = \ker(k)$, which are both the same due to exactness of $\E$. All of this can be summarized in the chain of inclusions
\begin{equation*}
    \{0\} = B^{(0)} \subseteq B^{(1)} \subseteq B^{(2)} \subseteq \cdots \subseteq j(D) = \im(j) = \ker(k) = k^{-1}(\{0\}) \subseteq \cdots \subseteq Z^{(2)} \subseteq Z^{(1)} \subseteq Z^{(0)} = E.
\end{equation*}
We define the quotient $E^{(r)} := Z^{(r)} / B^{(r)}$ and define the \textbf{$r$-th derived couple} to be the exact couple
\begin{center}
    \begin{tikzcd}[row sep=25pt]
        \E^{(r)}\colon &i^r(D) \ar[rr,"i^{(r)}"] &&i^r(D) \ar[dl,"j^{(r)}"] \\
        &&E^{(r)} \ar[ul,"k^{(r)}"]
    \end{tikzcd}
\end{center}
where $i^{(r)}$ denotes the restriction of $i$ to $i^r(D)$, $k^{(r)}([\alpha]):=k(\alpha)$, and $j^{(r)}$ maps an element $i^r(x)$ to $[j(x)]$. For the sake of completeness we include the following statement, whose proof consists mostly of straightforward computations and verifications.

\begin{proposition}
    The maps $i^{(r)}, j^{(r)}, k^{(r)}$ are well-defined and have bidegrees $(1,-1),(-(a+r),a+r), (-1,0)$, respectively. Moreover, $\E^{(r)}$ is an exact couple (of order $a+r$).
\end{proposition}

\begin{proof}
    %Showing the maps are well-defined
    We start by showing that the maps are well-defined. The map $i^{(r)}$ is well-defined since the restriction of $i$ to $i^r(D)$ clearly maps into $i^r(D)$. For $k^{(r)}$, we first need to show that it does not depend on the chosen representative. For this it suffices to show that $k(\beta)=0$ for any $\beta \in B^{(r)}$, which holds since $B^{(r)} \subseteq \im(j) = \ker(k)$. The fact that the image of $k^{(r)}$ lies inside $i^r(D)$ follows directly from the definition of $Z^{(r)}$. For $j^{(r)}$ we have to check that its definition does not depend upon the choice of a preimage under $i^r$. This follows from the fact that if $i^rx=0$, then $x \in \ker(i^r)$, thus $j(x) \in j(\ker(i^r)) = B^{(r)}$. On the other hand, $j(x) \in \im(j) = \ker(k) \subseteq Z^{(r)}$.

    %Bidegrees
    The bidegree of $i^{(r)}$ is clearly the same as the bidegree of $i$, as it is just the restriction. Similarly, the bidegree of $k^{(r)}$ is the same as the bidegree of $k$, since it is the induced map on a subquotient. As for $j^{(r)}$, since it is defined by a combination of the map $j$ and taking an $r$-fold preimage along the map $i$, it's bidegree equals the bidegree of $j$ minus $r$ times the bidegree of $i$, which yields $(-a-r,a+r)$.

    %Exactness 1: images contained in kernels
    To show the exactness of $\E^{(r)}$, we need to show three equalities, which boil down to six inclusions. The inclusions $\im(k^{(r)})\subseteq \ker(i^{(r)})$, $\im(i^{(r)})\subseteq \ker(j^{(r)})$, and $\im(j^{(r)})\subseteq \ker(k^{(r)})$ follow directly from $\im(k)\subseteq \ker(i)$, $\im(i)\subseteq \ker(j)$, and $\im(j)\subseteq \ker(k)$. Let us now show the other three inclusions.

    %Exactness 2.1
    To show that $\ker(i^{(r)}) \subseteq \im(k^{(r)})$, let $x\in i^r(D)$ with $ix=0$. Because $\E$ is exact, there exists $\alpha \in E$ such that $k(\alpha)=x$. By definition of $Z^{(r)}$, it follows that $\alpha \in Z^{(r)}$, thus $x = k(\alpha) = k^{(r)}([\alpha]) \in \im(k^{(r)})$.

    %Exactness 2.2
    To show that $\ker(j^{(r)}) \subseteq \im(i^{(r)})$, consider $i^rx \in i^r(D)$ with $j^{(r)}(i^r(x))=0$, i.e. $j(x)=0$. By exactness of $\E$, $x \in \im(i)$, which implies that $i^r(x) \in i^{r+1}(D)=\im(i^{(r)})$.

    %Exactness 2.3
    To show that $\ker(k^{(r)}) \subseteq \im(j^{(r)})$, consider $[\alpha] \in E^{(r)}$ with $k^{(r)}([\alpha])=0$. This means that $k(\alpha)=0$, so, by exactness of $\E$, $\alpha =j(x)$ for some $x \in D$. Therefore $[\alpha]=[j(x)] = j^{(r)}(i^r(x))$, hence $[\alpha] \in \im(j^{(r)})$. 
\end{proof}

We want to show that our direct definition of the $r$-th derived couple yields the same result as applying the definition of the derived couple $r$ times. For this, we first need an algebraic lemma, which could be summarized by saying that a subquotient of a subquotient is a subquotient.

\begin{lemma}[Proposition  2.1 of \cite{AtiyahMacDonald1994IntroToCommAlg}]\label{lem:subquot-explicit}
    Let $Z$ be a vector space with subspaces $B \subseteq X' \subseteq X \subseteq Z$. Then, $X/X' \cong (X/B) / (X'/B)$, where the isomorphisms are explicitly given by
    \[
    \begin{array}{rcl}
        X/X' &\cong &(X/B) / (X'/B) \\
        {[x]_{X'}} &\overset{\varphi}{\longmapsto} &[[x]_B]_{X'/B} \\
        {[x]_{X'}} &\overset{\psi}{\text{\reflectbox{$\longmapsto$}}} &[[x]_B]_{X'/B}
    \end{array}
    \]
    for all $x \in X$.
\end{lemma}

Now we are ready to prove the statement we mentioned before. Also this is not a new result, compare e.g. Exercise 5.9.1 in \cite{weibel1994introduction}. Again, we include a proof for convenience.

\begin{lemma}\label{lem:iterated-derived-couples}
    We have $\E' = \E^{(1)}$ and $(\E^{(r)})' \cong \E^{(r+1)}$.
\end{lemma}

\begin{proof}
    To see why $\E' = \E^{(1)}$, note that
    \begin{align*}
        \ker(d) &= k^{-1}(\ker(j)) = k^{-1}(\im(i)) = Z^{(1)}, \\
        \im(d) &= j(\im(k)) = j(\ker(i)) = B^{(1)},
    \end{align*}
    and thus $E' = \ker(d)/\im(d) = Z^{(1)} / B^{(1)} = E^{(1)}$. To see that also the maps agree, note that $i'$ and $i^{(1)}$ are both the restriction of $i$ to $i(D)$, $j'=j^{(1)}$ by definition and the same holds for $k'=k^{(1)}$.
    
    Let us now show that $(\E^{(r)})' \cong \E^{(r+1)}$. In other words, we want to show that the two exact couples 
    \begin{center}
    \begin{tikzcd}[row sep=25pt]
        (\E^{(r)})'\colon\hspace{-4ex} &i(i^r(D)) \ar[rr,"(i^{(r)})'"] &&i(i^r(D)) \ar[dl,"(j^{(r)})'"] 
        &\E^{(r+1)}\colon\hspace{-4ex} &i^{r+1}(D) \ar[rr,"i^{(r+1)}"] &&i^{r+1}(D) \ar[dl,"j^{(r+1)}"] \\
        &&(E^{(r)})' \ar[ul,"(k^{(r)})'"]
        &&&&E^{(r+1)} \ar[ul,"k^{(r+1)}"]
    \end{tikzcd}
\end{center}
    are isomorphic. For the upper row in both diagrams we have equality since $i(i^r(D))=i^{r+1}(D)$ and both $(i^{(r)})'$ and $i^{(r+1)}$ are simply the restriction of $i$. To see why $(E^{(r)})' \cong E^{(r+1)}$, note that $(E^{(r)})' = \ker(d^{(r)})/\im(d^{(r)})$, where $d^{(r)} = j^{(r)} \circ k^{(r)}$ is the differential on $E^{(r)}$. The kernel of $d^{(r)}$ can be expressed as follows.
    \begin{align*}
        \ker(d^{(r)}) &= \{[\alpha] \in Z^{(r)}/B^{(r)} \mid d^{(r)}([\alpha]) = 0 \} \\
        &= \{[\alpha] \in Z^{(r)}/B^{(r)} \mid \exists x \in D \text{ with } k(\alpha) = i^r(x) \text{ and } j(x) \in B^{(r)}\} \\
        &= \{[\alpha] \in Z^{(r)}/B^{(r)} \mid \exists x \in D \text{ with } k(\alpha) = i^r(x) \text{ and } j(x)=0\} \\
        %&= \{[\alpha] \in Z^{(r)}/B^{(r)} \mid \exists x \in D \text{ with } k(\alpha) = i^r(x) \text{ and } x \in \ker(j)\} \\
        %&= \{[\alpha] \in Z^{(r)}/B^{(r)} \mid \exists x \in D \text{ with } k(\alpha) = i^r(x) \text{ and } x \in \im(i)\} \\
        &= \{[\alpha] \in Z^{(r)}/B^{(r)} \mid \exists x \in D \text{ with } k(\alpha) = i^{r+1}(x)\} \\
        &= \{[\alpha] \in Z^{(r)}/B^{(r)} \mid \alpha \in k^{-1}(\im(i^{r+1}))\} \\ 
        &= \{[\alpha] \in Z^{(r)}/B^{(r)} \mid \alpha \in Z^{(r+1)} \} = Z^{(r+1)}/B^{(r)}.
    \end{align*}
    To see why the third equality holds, note that if $j(x) \in B^{(r)} = j(\ker(i^r))$, then this means that there exists $y \in \ker(i^r)$ such that $j(x) = j(y)$. If we now consider the element $x' = x-y$, then we have that $i^r(x') = i^r(x) + i^r(y) = i^r(x) = k(\alpha)$ and $j(x') = j(x)-j(y) = 0$. Replacing $x$ by $x'$ yields the equality. For the fourth equality, note that $j(x)=0$ if and only if $x\in \ker j$, which by exactness is equal to $\im i$, so  $x\in \im i^r$ and $j(x)=0$ if and only if $x\in \im i^{r+1}$. All other equalities are the result of unpacking or restating definitions. We continue with the image of $d^{(r)}$:
    \begin{align*}
        \im(d^{(r)}) &= \{[\alpha] \in Z^{(r)}/B^{(r)} \mid \exists [\beta] \in Z^{(r)}/B^{(r)} \text{ such that } [\alpha] = d^{(r)}([\beta]) \} \\ 
        %&= \{[\alpha] \in Z^{(r)}/B^{(r)} \mid \exists \beta \in Z^{(r)},\ x \in D \text{ such that } k(\beta) = i^r(x) \text{ and } [j(x)]= [\alpha] \} \\ 
        &= \{[\alpha] \in Z^{(r)}/B^{(r)} \mid \exists x \in D \text{ such that } i^r(x) \in \im(k) \text{ and } [j(x)]= [\alpha] \} \\ 
        %&= \{[\alpha] \in Z^{(r)}/B^{(r)} \mid \exists x \in D \text{ such that } i^r(x) \in \ker(i) \text{ and } [j(x)]= [\alpha] \} \\
        &= \{[\alpha] \in Z^{(r)}/B^{(r)} \mid \exists x \in \ker(i^{r+1}) \text{ such that } [j(x)]= [\alpha] \} \\
        &= \{[\alpha] \in Z^{(r)}/B^{(r)} \mid \exists x \in \ker(i^{r+1}) \text{ such that } j(x)= \alpha \} \\
        &= \{[\alpha] \in Z^{(r)}/B^{(r)} \mid \alpha \in j(\ker(i^{r+1})) \} \\
        &= \{[\alpha] \in Z^{(r)}/B^{(r)} \mid \alpha \in B^{(r+1)} \} = B^{(r+1)}/B^{(r)}.
    \end{align*}
    To see why the third equality holds, note that $\im(k)=\ker(i)$ by exactness and that $i^r(x) \in \ker(i)$ means exactly that $i^{r+1}(x)=0$, i.e. $x \in \ker(i^{r+1})$. For the fourth equality, note that if $[j(x)]=[\alpha]$, this means that $j(x)-\alpha \in B^{(r)} = j(\ker(i^r))$, i.e. there exists $y \in \ker(i^r)$ such that $j(x)-\alpha=j(y)$. If we consider the element $x' = x-y$, then we have $i^{r+1}(x') = i^{r+1}(x) + i^{r+1}(y) =0$ and $j(x')=j(x)-j(y)=\alpha$, thus by replacing $x$ with $x'$ we get the equality. Again, all other equalities are the result of reformulations of different definitions. We thus have
    \[
    (E^{(r)})' = \ker(d^{(r)})/\im(d^{(r)}) = (Z^{(r+1)}/B^{(r)}) / (B^{(r+1)}/B^{(r)}) \cong Z^{(r+1)}/B^{(r+1)} = E^{(r+1)},
    \]
    where the isomorphism follows from \Cref{lem:subquot-explicit}. To see that the two exact couples are isomorphic, it remains to show that the diagram
    \begin{center}
    \begin{tikzcd}[row sep=25pt, column sep=45pt]
        i(i^r(D)) \ar[d,equal] \ar[r,"(j^{(r)})'"] &(E^{(r)})' = (Z^{(r+1)}/B^{(r)}) / (B^{(r+1)}/B^{(r)}) \ar[d,"\psi"] \ar[r,"(k^{(r)})'"] &i(i^r(D)) \ar[d,equal] \\
        i^{r+1}(D) \ar[r,"j^{(r+1)}"] &E^{(r+1)} = Z^{(r+1)}/B^{(r+1)} \ar[r,"k^{(r+1)}"] &i^{r+1}(D)
    \end{tikzcd}
    \end{center}
    commutes. The map $\psi$ in the middle is isomorphism from \Cref{lem:subquot-explicit} and it can be described by saying that it maps $[[\alpha]_{B^{(r)}}]_{B^{(r+1)}/B^{(r)}} \mapsto [\alpha]_{B^{(r+1)}}$. The right square commutes since 
    \[
    (k^{(r)})'([[\alpha]_{B^{(r)}}]_{B^{(r+1)}/B^{(r)}}) = k^{(r)}([\alpha]_{B^{(r)}}) = k(\alpha) = k^{(r+1)}([\alpha]_{B^{(r+1)}})
    \]
    and the left square commutes because
    \[
    (j^{(r)})'(i(i^r(x))) = [j^{(r)}(i^r(x))]_{B^{(r+1)}/B^{(r)}} = [[j(x)]_{B^{(r)}}]_{B^{(r+1)}/B^{(r)}} \overset{\psi}{\longmapsto} [j(x)]_{B^{(r+1)}} = j^{(r+1)}(i^{r+1}(x)).
    \]
    This completes the proof.
\end{proof}

If in an exact couple we ignore $D$ and focus only on $E$, $E^{(1)}$, $E^{(2)}$, ..., each endowed with the differential defined as $d^{(r)} := j^{(r)} \circ k^{(r)}$, then this yields a spectral sequence. This is stated in the result \cite[Proposition 5.9.2]{weibel1994introduction}, which we recall here without proof.

\begin{proposition}\label{prop:weibel}
    An exact couple $\E$ in which $i$, $j$, and $k$ have bidegrees $(1,-1)$, $(-a,a)$, and $(-1, 0)$ determines a spectral sequence $\{(E^{(r)}_{p,q}, d^{(r)})\}_{r\ge a+1}$ with $d^{(r)}:= j^{(r)} \circ k^{(r)}$ starting at page $a+1$. A morphism of exact couples induces a morphism of the corresponding spectral sequences.
\end{proposition}

An exact couple is called \textbf{bounded below} if for each $n$ there exists an integer $f(n)$ such that for $p<f(p+q)$ we have $D_{p,q}=0$. 
The next result is \cite[Classical Convergence Theorem 5.9.7]{weibel1994introduction}. In \Cref{sec:spectral-sequence-of-filtration} we will use it to show that the infinity-page of the spectral sequence of a topological filtration contains information about the homology of the union of all the levels in the filtration.

\begin{theorem}\label{thm:weibel}
    If an exact couple is bounded below, then the spectral sequence is bounded below and converges to $(H_n)_{n \in \Z}$, where $\displaystyle H_n = \lim_{p\to \infty} D_{p,n-p}$ is the direct limit along the maps $i_{p,n-p}\colon D_{p,n-p} \to D_{p+1,n-p-1}$.
\end{theorem}

\subsection{The chain complex of a spectral sequence (via exact couples)}\label{sec:chain-complex-of-spectral-sequence}

We now describe a way to reorganize the algebraic information contained in a spectral sequence to obtain a chain complex, keeping some important properties of the spectral sequence intact. In particular, the vector spaces in the chain complex are isomorphic to the direct sums along the diagonals of the first page and its homology vector spaces are isomorphic to the direct sums along the diagonals of the infinity-page.

Let us consider a spectral sequence $E = \{E^r_{\bullet,\bullet}\}_{r\ge a+1}$ and let us assume that $E$ is induced from an exact couple as in \Cref{prop:weibel}
\begin{center}
    \begin{tikzcd}[row sep=25pt]
        \E\colon &D \ar[rr,"i"] &&D \ar[dl,"j"] \\
        &&E=E^a, \ar[ul,"k"]
    \end{tikzcd}
\end{center}
where the map $j$ has bidegree $(-a,a)$. Recall that $i$ has bidegree $(1,-1)$ and $k$ has bidegree $(-1,0)$. We usually hide the bigradings and bidegrees from the notation and make them explicit only when needed.

We next describe a method to construct, from such a spectral sequence $E$, a chain complex $C_\bullet=C_\bullet(E)$ complex with the following  properties: {\em (i)} $C_n \cong \bigoplus_{p+q=n} E_{p,q}$, and {\em (ii)} $H_n(C_\bullet) \cong H_n = \lim_{p\to \infty} D_{p,n-p}$. 

In \Cref{sec:spec-seq} we have seen the sequence 
\[
0 = B^{(0)} \subseteq B^{(1)} \subseteq B^{(2)} \subseteq \cdots \subseteq Z^{(2)} \subseteq Z^{(1)} \subseteq Z^{(0)} = E,
\]
of subspaces of $E$, defined by
\begin{align*}
    B^{(r)} &:= j(\ker(i^r)),
    &Z^{(r)} &:= k^{-1}(\im(i^r)).
\end{align*}
As $E$ is bigraded, its subspaces $B^{(r)}$ and $Z^{(r)}$ are also bigraded, but we usually hide this from the notation.
Due to finite total dimension, the sequences of $B^{(r)}$ and $Z^{(r)}$ both stabilize, i.e. there exists a number $n$ such that $B^{(r)} = B^{(r+1)}$ and $Z^{(r)} = Z^{(r+1)}$ for all $r\ge n$. This filtration of subspaces induces a (non-canonical) decomposition
\begin{equation}\label{eq:first-page-filtration-iso}
    E^a \cong (B^{(1)}/B^{(0)}) \oplus \cdots \oplus (B^{(n)}/B^{(n-1)}) \oplus (Z^{(n)}/B^{(n)}) \oplus (Z^{(n-1)}/Z^{(n)}) \oplus \cdots \oplus (Z^{(0)}/Z^{(1)}).
\end{equation}
The term non-canonical in this case refers to the isomorphism, for which there may be multiple candidates with no obvious preferred choice. We will use this decomposition together with the following result in order to define a chain complex.

Here is a result that we will need.

\begin{proposition}
    The differentials $d^{(r)}\colon E^{(r)} \to E^{(r)}$ induce isomorphisms $\overline{d}^{(r)}\colon Z^{(r)} / Z^{(r+1)} \overset{\cong}{\longrightarrow} B^{(r+1)} / B^{(r)}$.
\end{proposition}

\begin{proof}
By standard linear algebra arguments, $d^{(r)}$ induces an isomorphism $\hat{d}^{(r)}\colon E^{(r)}/\ker(d^{(r)}) \cong \im(d^{(r)})$. In the proof of \Cref{lem:iterated-derived-couples} we have seen that $\ker(d^{(r)}) = Z^{(r+1)}/B^{(r)}$ and $\im(d^{(r)}) = B^{(r+1)}/B^{(r)}$. Together with the isomorphism $\varphi$ from \Cref{lem:subquot-explicit} we get
\begin{center}
    \begin{tikzcd}
        Z^{(r)} / Z^{(r+1)} \ar[r,"\varphi"] \ar[rr,bend right=20,"=:\overline{d}^{(r)}"] & (Z^{(r)}/B^{(r)})/(Z^{(r+1)}/B^{(r)})   \ar[r,"\hat{d}^{(r)}"] & B^{(r+1)}/B^{(r)},
    \end{tikzcd}
\end{center}
which concludes the proof.
\end{proof}

\begin{proposition}\label{prop:chain-complex-of-spectral-sequence}
    Given a spectral sequence $E$ induced from an exact couple $\E$, we can construct its associated chain complex $C_\bullet(E)$ as follows. We first consider the bigraded vector space
    \[
    (B^{(1)}/B^{(0)}) \oplus \cdots \oplus (B^{(n)}/B^{(n-1)}) \oplus (Z^{(n)}/B^{(n)}) \oplus (Z^{(n-1)}/Z^{(n)}) \oplus \cdots \oplus (Z^{(0)}/Z^{(1)})
    \]
    and endow it with the differential $d$, which is defined on each summand according to the following matrix in block form with $(2n+1)\times(2n+1)$ blocks. (total boundary matrix)
    \[
    \left(\begin{array}{ccc|c|ccc}
	& &    	&0     & &0&\overline{d}^{(n-1)}  \\
    &0 &&\vdots &&\text{\rotatebox{70}{$\ddots$}} \\
    &&&0 &\overline{d}^{(0)}& 0\\ \hline
	0 &\cdots &0 &0 &0 &\cdots &0\\ \hline
     && &0 & && \\		
    &0 & &\vdots &&0 & \\
    &&&0
	\end{array}\right)
    \]
    More explicitly,
    \[
    C_n := \bigoplus_{p+q=n} (B_{p,q}^{(1)}/B_{p,q}^{(0)}) \oplus \cdots \oplus (B_{p,q}^{(n)}/B_{p,q}^{(n-1)}) \oplus (Z_{p,q}^{(n)}/B_{p,q}^{(n)}) \oplus (Z_{p,q}^{(n-1)}/Z_{p,q}^{(n)}) \oplus \cdots \oplus (Z_{p,q}^{(0)}/Z_{p,q}^{(1)})
    \]
    and
    \[
    \partial_n := \bigoplus_{p+q=n} \bigoplus_{r=0}^{n-1} \overline{d}^{(r)}_{p,q} .
    \]
\end{proposition}

\begin{proof}
    We need to check that this is indeed a chain complex, i.e. that the differential squares to zero. In the composition $\partial_{n-1} \circ \partial_n$, anything that could potentially be non-zero is given by compositions of the maps $\overline{d}^{(r)}_{p,q}$, which are induced from the differentials on $E^{(r)}$. These are known to square to zero, implying that $\partial_{n-1} \circ \partial_n =0$.
\end{proof}

\begin{proposition}
    Isomorphic exact couples ${\mathcal E}$ and ${\mathcal E}'$ lead to isomorphic spectral sequences $E$ and $E'$, and hence to isomorphic chain complexes $C_\bullet(E)$ and $C_\bullet(E')$.
\end{proposition}

\begin{proof}
    An isomorphism between exact couples induces isomorphisms between their derived couples and hence between their corresponding spectral sequences. An isomorphism between two spectral sequences $E$ and $E'$ restricts to isomorphisms between the corresponding subspaces $B^{(r)} \cong {B'^{(r)}}$ and $Z^{(r)} \cong {Z'^{(r)}}$, thus all of the subquotients involved in the definition of the chain complex are isomorphic and since the differential is induced by the differentials of the spectral sequences, these are isomorphic as well.
\end{proof}

\begin{proposition}\label{prop:chain-complex-from-spectral-sequence-homology}
    Let $C_\bullet(E)$ be the chain complex associated to a spectral sequence $E$ as in \Cref{prop:chain-complex-of-spectral-sequence}. If $E$ converges to $H_*$, then the homology of $C_\bullet(E)$ is isomorphic to $H_*$.     
\end{proposition}

\begin{proof}
    By \Cref{prop:ss-converges-implies-isomorphism} it is enough to show that $H_n(C_\bullet(E)) \cong \bigoplus_{p+q=n} E^\infty_{p,q}$. Since the differential on $C_\bullet(E)$ is given in block form, where all the non-zero blocks are isomorphisms, we see that the homology is exactly $Z^{(n)} / B^{(n)}$, which is the only block that is not affected by the differential. By the definition of $Z^{(n)}$ and $B^{(n)}$, this quotient is isomorphic to the infinity-page of $E$.
\end{proof}

\section{The spectral sequence of the unstable manifolds filtration}

In this section we describe the tools that are used to transform dynamical/topological information into algebraic information. In \Cref{sec:cech-homology}, we collect some facts about \v{C}ech homology. In \Cref{sec:homology-computations} we make a few homology computations, which we then use in \Cref{sec:canonical-generators} to understand the relative \v{C}ech homology of two adjacent levels in the filtration described in \Cref{sec:filtration-induced-by-vector-field}. In \Cref{sec:spectral-sequence-of-filtration}, we prove that a topological filtration of compact spaces induces a spectral sequence in \v{C}ech homology.

\subsection{Preliminaries on \v{C}ech homology}\label{sec:cech-homology}

All the topological spaces are assumed to be Hausdorff and we usually do not state this explicitly. The spaces we consider are subsets of manifolds or quotients of Hausdorff spaces by compact subsets and thus again Hausdorff. 

Given a pair of spaces $(X,A)$ we write $\check{H}_k(X,A)$ for its \v{C}ech homology in degree $k$ with coefficients in a field $\Fi$. We write $\check{H}_k(X)$ for $\check{H}_k(X,\emptyset)$. This is defined by taking the inverse limit of the simplicial homology of the nerves of all open coverings of $(X,A)$. See \cite{FoundationsAlgTop} for details. We recall some definitions and list some properties that we are going to use, giving references for proofs. 

The main reason why we use \v{C}ech homology instead of singular homology is that it satisfies a stronger version of the excision axiom, namely invariance under relative homeomorphisms. From this it follows that the \v{C}ech homology for an arbitrary compact pair $(X,A)$ is isomorphic to the reduced one of the quotient space $X/A$. Singular homology has this property only for nice enough pairs ($A$ needs to be a deformation retract of some neighbourhood in $X$), but the pairs appearing in the filtration assigned to a vector field are not always nice in that sense.

\begin{definition}
    A map of pairs $f\colon (X,A) \to (Y,B)$ is called a \textbf{relative homeomorphism} if it maps $X\setminus A$ homeomorphically to $Y\setminus B$.
\end{definition}

The following result is \cite[Chapter X, Lemma 5.2]{FoundationsAlgTop}.

\begin{lemma}
    If $f\colon (X,A) \to (Y,B)$ is a map of compact pairs and $f$ maps $X\setminus A$ bijectively onto $Y\setminus B$, then $f$ is a relative homeomorphism.
\end{lemma}

As an immediate corollary we get that the quotient map that collapses a closed subspace $A$ of a compact space $X$ to a point is a relative homeomorphism.  

\begin{corollary}\label{cor:quotient-rel-homeo}
    If $(X,A)$ is a compact pair, then the quotient map $q\colon (X,A) \to (X/A,\{*\})$ is a relative homeomorphism.
\end{corollary}

The following result is \cite[Chapter X, Theorem 5.4]{FoundationsAlgTop}, stating that \v{C}ech homology is invariant under relative homeomorphisms.

\begin{theorem}\label{thm:rel-homeo-induce-iso}
    If $f\colon (X,A) \to (Y,B)$ is a relative homeomorphism between compact pairs, then the induced map $f_*\colon \check{H}_*(X,A) \to \check{H}_*(Y,B)$ is an isomorphism.
\end{theorem}

A \textbf{pointed space} $(X,*)$ consists of a topological space $X$ together with a chosen point $* \in X$, called the \textbf{basepoint}. Usually we denote the basepoints of all pointed spaces with the same symbol $*$, unless it is helpful to do otherwise. When considering a quotient space $X/A$, if not otherwise specified, the canonical basepoint is the image of $A$ under the quotient map $X\to X/A$.

The \textbf{wedge sum} $X_1 \vee X_2$ of two pointed spaces $(X_1,*_1)$ and $(X_2,*_2)$ is defined as the quotient space $(X_1\sqcup X_2)/(\{*_1\}\sqcup\{*_2\})$. One can check that this operation is associative, i.e. $(X_1\vee X_2)\vee X_3$ and $X_1 \vee (X_2\vee X_3)$ are naturally homeomorphic. Thus we can define the $n$-fold wedge sum by iteration, with the according choice of the canonical basepoint.

\begin{proposition}\label{prop:homology-iso-wedge-directsum}
    Given compact pointed spaces $(X_1,*_1),\ldots,(X_n,*_n)$, we have 
    \[
    \check{H}_k(X_1 \vee \cdots \vee X_n,*) \cong \check{H}_k(X_1,*_1) \oplus \cdots \oplus \check{H}_k(X_n,*_n).
    \]
\end{proposition}

\begin{proof}
    By an induction argument, it is enough to prove the statement for $n=2$. Consider the quotient map $q\colon (X_1 \sqcup X_2, \{*_1\}\sqcup \{*_2\}) \to (X_1 \vee X_2,*)$. By definition of the wedge sum, $q$ is a relative homeomorphism and hence by \Cref{thm:rel-homeo-induce-iso} induces an isomorphism in \v{C}ech homology. When considering \v{C}ech homology of the disjoint union $X_1\sqcup X_2$, we can restrict to open covers where every open set is a subset of either $X_1$ or $X_2$, since the set of such covers is cofinal \cite[Chapter VIII, Corollary 3.16]{FoundationsAlgTop}. The nerves of such covers are disjoint unions of the corresponding nerves of the covers for $X_1$ and $X_2$ and thus the simplicial homology decomposes as a direct sum. Since inverse limits commute with direct sums, this property stays true also when taking the inverse limit and it follows that $\check{H}_k(X_1\sqcup X_2,\{*_1\} \sqcup \{*_2\}) \cong \check{H}_k(X_1,\{*_1\}) \oplus \check{H}_k(X_2,\{*_2\})$. 
\end{proof}

For nice enough pairs $(X,A)$, \v{C}ech homology agrees with singular homology. Nice enough in this case means \textbf{triangulable}, i.e. there exists a simplicial pair $(K,L)$ and a homeomorphism $h\colon (|K|,|L|) \to (X,A)$.
We now state two results \cite[Chapter IX, Theorem 9.3]{FoundationsAlgTop} and \cite[Chapter VII, Theorem 10.1]{FoundationsAlgTop}, which combined tell us that such a triangulation induces isomorphisms from the \v{C}ech homology of $(X,A)$ to the simplicial homology of $(K,L)$ and from there to the singular homology of $(X,A)$.

\begin{theorem}
    A triangulation $h\colon (|K|,|L|) \to (X,A)$ induces a covering of $(X,A)$ which yields isomorphisms
    \[
    \check{H}_q(X,A;G) \overset{\cong}{\longrightarrow} H_q^\Delta(K,L;G)
    \]
    between the \v{C}ech homology of $(X,A)$ and the simplicial homology of $(K,L)$ for any coefficient group $G$, for all $q$. These isomorphisms commute with the connecting homomorphisms for the long exact sequences of the pairs $(K,L)$ and $(X,A)$, respectively.
\end{theorem}

\begin{theorem}
    A triangulation $h\colon (|K|,|L|) \to (X,A)$ induces isomorphisms 
    \[
    H_q^\Delta(K,L;G) \overset{\cong}{\longrightarrow} H_q(X,A;G)
    \]
    from the simplicial homology of $(K,L)$ to the singular homology of $(X,A)$ over any coefficient group $G$, for all $q$. These isomorphisms commute with the connecting homomorphisms for the long exact sequences of the pairs $(K,L)$ and $(X,A)$, respectively.
\end{theorem}

By combining these two results, we get that \v{C}ech homology is isomorphic to singular homology for triangulable pairs.

\begin{corollary}\label{cor:cech-sing-iso}
    If $(X,A)$ is a triangulable pair, then there is a natural isomorphism $H_*(X,A) \cong \check{H}_*(X,A)$, commuting with the connecting homomorphisms from the long exact sequences.
\end{corollary}

Denote by $X^\star$ the \textbf{one-point compactification} of a space $X$. 
See either \cite{FoundationsAlgTop} or \cite{MunkresTopology} for the precise definition. 
When considering $X^\star$ as a pointed space, we choose the newly added point $*$, often called the \textbf{point at infinity}, as the basepoint. 
We repeat the statement \cite[Chapter X, Lemma 6.3]{FoundationsAlgTop}, which gives a criterion to check when a space is homeomorphic to a one-point compactification.

\begin{lemma}\label{lem:cts-extension-rel-homeo}
    Let $X,Y$ be locally compact, $A\subseteq X$ a closed subset and $f\colon X\setminus A \to Y$ a proper map (i.e. inverse images of compact sets under $f$ are compact). Define $\overline{f}\colon X \to Y^\star$ by setting $\overline{f}(x)=f(x)$ for $x \in X\setminus A$ and $f(a)=*$ for $a \in A$. Then, $\overline{f}$ is a continuous extension of $f$. 
    If $f$ is a homeomorphism, then $\overline{f}$ is a relative homeomorphism between the pairs $(X,A)$ and $(Y^\star,\{*\})$.
    If, moreover, $X$ is compact, and $A$ is a single point, then $\overline{f}$ is a homeomorphism.
\end{lemma}

Using this lemma, we can prove that for compact pairs $(X,A)$, collapsing $A$ to a point is the same as first removing $A$ and then adding a point at infinity.

\begin{proposition}\label{prop:quotient-compactification}
    If $(X,A)$ is a pair of compact Hausdorff spaces, then $X/A \cong  (X\setminus A)^\star$.
\end{proposition}

\begin{proof}
    Let $f\colon X\setminus A \to X\setminus A$ be the identity map, where the domain is considered as a subspace of $X/A$ and the target as a subspace of $X$. By tracing the definitions of the subset topology and the quotient topology, one can verify that the open sets are the same in both spaces, namely open subsets of $X$ disjoint from $A$ (since $A$ is closed), and thus $f$ is a homeomorphism. Since $X$ is compact, also $X/A$ is compact. Therefore, it follows from \Cref{lem:cts-extension-rel-homeo} that the map $\overline{f}\colon X/A \to (X\setminus A)^\star$, which is the identity on $X\setminus A$ and sends the point $A$ to $*$, is a homeomorphism.
\end{proof}

A disadvantage of using \v{C}ech homology is that the long sequence of homology groups of a pair $(X,A)$ may not be exact. By \Cref{cor:cech-sing-iso}, the sequence is exact for a triangulable pair, but we are interested also in pairs that are not triangulable. Luckily, since we are working with closed sets of compact manifolds, we are only considering compact pairs. Thanks to this and to the fact that we are working with field coefficients, we do indeed get long exact sequences of pairs, as stated in \cite[Chapter IX, Theorem 7.6]{FoundationsAlgTop}, which we repeat here for convenience.

\begin{theorem}\label{thm:Cech-homology-exactness-axiom}
    If $(X,A)$ is a compact pair and $G$ is a vector space over a field $\Fi$, then the long sequence in \v{C}ech homology with coefficients in $G$ associated to the pair $(X,A)$ is exact.
\end{theorem}

\subsection{The spectral sequence in \v{C}ech homology of a topological filtration}\label{sec:spectral-sequence-of-filtration}

It is well-known that every filtration of a topological space induces a spectral sequence in homology. We reprove this result for \v{C}ech homology in the case of compact spaces, using the preliminary results from \Cref{sec:cech-homology}. Note that additional to the spaces in the filtration being compact, we are making use of the fact that we are working with field coefficients, in order to guarantee the existence of the long exact sequences in \v{C}ech homology.

\begin{proposition}\label{prop:spectral-sequence-from-filtration}
    If $\emptyset \subseteq L_0 \subseteq L_1 \subseteq \cdots \subseteq L_n = M$ is a filtration of a compact topological space $M$, such that each $L_p$ is compact, then this determines a spectral sequence $\{E_{\bullet,\bullet}^r\}_{r\ge 1}$ whose first page is given by $E^1_{p,q} = \check{H}_{p+q}(L_p,L_{p-1})$ and that converges to $\check{H}_*(M)$.
\end{proposition}

\begin{proof}
    Let $n:=p+q$. Since we are using field coefficients and the spaces $L_p$ are compact, by \Cref{thm:Cech-homology-exactness-axiom}, for each pair $(L_p,L_{p-1})$ there exists a long exact sequence in \v{C}ech homology 
    \begin{center}
        \begin{tikzcd}
            \cdots \ar[r] &\check{H}_n(L_{p-1}) \ar[r,"i"] &\check{H}_n(L_p) \ar[r,"j"] &\check{H}_n(L_p,L_{p-1}) \ar[r,"k"] &\check{H}_{n-1}(L_{p-1}) \ar[r] &\cdots,
        \end{tikzcd}
    \end{center}
    where $i$ is the map induced in \v{C}ech homology from the inclusion $L_{p-1} \hookrightarrow L_p$, $j$ is the map induced in \v{C}ech homology from the inclusion of pairs $(L_p,\emptyset) \hookrightarrow (L_p,L_{p-1})$, and $k$ is the connecting homomorphism from \v{C}ech homology. These maps of course depend on $n$ and $p$ and we will endow them with indices later.
    Considering all of these long exact sequences at once, we can arrange them next to each other in the following way. Note that the long exact sequence of the pair $(L_0,\emptyset)$ consists only of the equalities $\check{H}_n(L_0,\emptyset) = \check{H}_n(L_0)$ since the homology of the empty set is zero.
\begin{center}
\begin{tikzcd}
    &\cdots \ar[drr,dotted] & &\cdots \ar[drr,dashed] \\
    \check{H}_1(L_0,\emptyset) &\check{H}_1(L_0) \ar[l, "="{above}] \ar[drr,dotted] &\check{H}_2(L_1,L_0) \ar[l,dotted] &\check{H}_2(L_1) \ar[l,dotted] \ar[drr,dashed] &\check{H}_3(L_2,L_1) \ar[l,dashed] &\check{H}_3(L_2) \ar[l,dashed] &\cdots \ar[l] \\
    \check{H}_0(L_0,\emptyset) &\check{H}_0(L_0) \ar[l, "="{above}] \ar[drr,dotted] &\check{H}_1(L_1,L_0) \ar[l,dotted] &\check{H}_1(L_1) \ar[l,dotted] \ar[drr,dashed] &\check{H}_2(L_2,L_1) \ar[l,dashed] &\check{H}_2(L_2) \ar[l,dashed] &\cdots \ar[l] \\
                       &0                 &\check{H}_0(L_1,L_0) \ar[l,dotted] &\check{H}_0(L_1) \ar[l,dotted] \ar[drr,dashed] &\check{H}_1(L_2,L_1) \ar[l,dashed] &\check{H}_1(L_2) \ar[l,dashed] &\cdots \ar[l] \\
                       &                  &                    &0                        &\check{H}_0(L_2,L_1) \ar[l,dashed] &\check{H}_0(L_2) \ar[l,dashed] &\cdots \ar[l]
\end{tikzcd}
\end{center}
    Let us formally extend this diagram to infinity in all directions by declaring that $L_p=\emptyset$ for all $p<0$ and $L_p=L_n$ for all $p>n$, and by using the convention that homology is zero in negative degrees.
    We now define $D_{p,q} = \check{H}_{p+q}(L_p)$ and $E_{p,q} = \check{H}_{p+q}(L_p,L_{p-1})$ and endow $i$, $j$, $k$ with indices such that
    \begin{align*}
        i_{p,q}\colon &\underbrace{\check{H}_{n}(L_p)}_{=D_{p,q}} \longrightarrow \underbrace{\check{H}_{n}(L_{p+1})}_{=D_{p+1,q-1}},  
        &j_{p,q}\colon &\underbrace{\check{H}_{n}(L_p)}_{=D_{p,q}} \longrightarrow \underbrace{\check{H}_{n}(L_p,L_{p-1})}_{=E_{p,q}}, \\[1ex]
        k_{p,q}\colon &\underbrace{\check{H}_{n}(L_p,L_{p-1})}_{=E_{p,q}} \longrightarrow \underbrace{\check{H}_{n-1}(L_{p-1})}_{=D_{p-1,q}}.
    \end{align*}
    Note that $D_{p,q}$ and $E_{p,q}$ are defined for all $p,q \in \Z$, but many of the terms are zero, namely $D_{p,q}=0$ if $p<0$ or $q<-p$, and $E_{p,q}=0$ if $p<0$ or $p>n$ or $q<-p$. In terms of $D_{p,q}$ and $E_{p,q}$ the diagram from before now looks as follows.
    \begin{center}
    \begin{tikzcd}
        &\cdots \ar[drr,dotted] & &\cdots \ar[drr,dashed] \\
        E_{0,1} &D_{0,1} \ar[l, "="{above}] \ar[drr,dotted] &E_{1,1} \ar[l,dotted] &D_{1,1} \ar[l,dotted] \ar[drr,dashed] &E_{2,1} \ar[l,dashed] &D_{2,1} \ar[l,dashed] &\cdots \ar[l] \\
        E_{0,0} &D_{0,0} \ar[l, "="{above}] \ar[drr,dotted] &E_{1,0} \ar[l,dotted] &D_{1,0} \ar[l,dotted] \ar[drr,dashed] &E_{2,0} \ar[l,dashed] &D_{2,0} \ar[l,dashed] &\cdots \ar[l] \\
                        &0                 &E_{1,-1} \ar[l,dotted] &D_{1,-1} \ar[l,dotted] \ar[drr,dashed] &E_{2,-1} \ar[l,dashed] &D_{2,-1} \ar[l,dashed] &\cdots \ar[l] \\
                        &                  &                    &0                        &E_{2,-2} \ar[l,dashed] &D_{2,-2} \ar[l,dashed] &\cdots \ar[l]
    \end{tikzcd}
    \end{center}
    In order to display all of this information more compactly, we define the vector spaces $D=\bigoplus_{p,q}D_{p,q}$ and $E=\bigoplus_{p,q}E_{p,q}$ and the linear maps $i=\bigoplus_{p,q}i_{p,q}$, $j=\bigoplus_{p,q}j_{p,q}$ and $k=\bigoplus_{p,q}k_{p,q}$. The result is an exact couple 
    \begin{center}
        \begin{tikzcd}[row sep=25pt]
            \E\colon &D \ar[rr,"i"] &&D \ar[dl,"j"] \\
            &&E \ar[ul,"k"]
        \end{tikzcd}
    \end{center}
    where $i,j,k$ have the bidegrees $(1,-1), (0,0), (-1,0)$, respectively. 
    The exactness of this couple follows immediately from the exactness of the long exact sequence in homology for each pair $(L_p,L_{p-1})$, since exactness is preserved under direct sums.
    By \Cref{prop:weibel} we thus get a spectral sequence starting at page $1$. Moreover, since this exact couple is bounded, we can apply \Cref{thm:weibel}, which tells us that the spectral sequence is bounded below and converges to the direct limit $\lim_{p\to \infty} D_{p,n-p}$. By the definition of the $D_{p,q}$, this is equal to $\check{H}_n(M)$.
\end{proof}

\subsection{Preparatory homology computations}\label{sec:homology-computations}

In this section we compute the relative \v{C}ech homology groups of the unstable manifolds of fixed points and closed orbits with respect to their boundaries. This will be helpful in the next section for endowing the relative \v{C}ech homology of the pairs coming from the filtration. Recall that by \Cref{lem:unstable-manifold-homeomorphic}, the unstable manifold of a fixed point of index $k$ is homeomorphic to $\R^k$ and the unstable manifold of a closed orbit of index $k$ is homeomorphic to $\R^k \times S^1$. 

We first treat the easier case of the unstable manifold of a fixed point.

\begin{corollary}\label{cor:relative-cech-homology-unstable-fixpt}
    If $\beta$ is a fixed point of index $k\ge 0$, then 
    \[
    \check{H}_r(\overline{W^u(\beta)},\partial W^u(\beta)) \cong \begin{cases}
        \Fi, &\text{if } $r=k$, \\
        0, &\text{otherwise}.
    \end{cases}
    \]
\end{corollary}

\begin{proof}
    This follows from
    \[
    \check{H}_r(\overline{W^u(\beta)},\partial W^u(\beta)) 
    \cong \check{H}_r(\overline{W^u(\beta)}/\partial W^u(\beta),*) 
    \cong \check{H}_r(W^u(\beta)^\star,*)
    \cong \check{H}_r(S^k,*)
    \cong H_r(S^k,*).
    \]
    The first and second isomorphisms follow from \Cref{thm:rel-homeo-induce-iso} together with \Cref{cor:quotient-rel-homeo} and \Cref{prop:quotient-compactification}, respectively. Note that $\overline{W^u(\beta)}$ is compact as it is a closed subset of $M$, which is compact. The third isomorphism holds by \Cref{lem:unstable-manifold-homeomorphic} and because $(\R^k)^\star \cong S^k$. The fourth isomorphism follows from \Cref{cor:cech-sing-iso}. 
\end{proof}

Now we want to do the analogous computation for the unstable manifold of a closed orbit. This is more difficult because the one-point compactification of $\R^k\times S^1$ is a more complicated space. We therefore use a different compactification, namely $D^k\times S^1$, where $D^k$ denotes the closed unit disk of dimension $k$. The "boundary" of this compactification is given by $S^{k-1}\times S^1$, so eventually we want to compute the relative homology of the pair $(D^k\times S^1,S^{k-1}\times S^1)$. We start by computing the homology of $S^{k-1}\times S^1$.

\pagebreak
\begin{lemma}\label{lem:homology-of-sphere} The following isomorphisms hold:
    \begin{enumerate}[(i)]
        \item\label{item:homology-of-sphere1} $H_r(S^0\times S^1) \cong \begin{cases}
            \Fi \oplus \Fi, &\text{if } r=0,1,\\
            0, &\text{otherwise}.
        \end{cases}$ 
        \item\label{item:homology-of-sphere2} $H_r(S^1\times S^1) \cong \begin{cases}
            \Fi, &\text{if } r=0,2,\\
            \Fi\oplus\Fi, &\text{if } r=1,\\
            0, &\text{otherwise}.
        \end{cases}$
        \item\label{item:homology-of-sphere3} For $k\ge 3$, we have $H_r(S^{k-1}\times S^1) \cong \begin{cases}
            \Fi, &\text{if } r=0,1,k-1,k,\\
            0, &\text{otherwise}.
        \end{cases}$
    \end{enumerate}
\end{lemma}

\begin{proof}
    For any $k\ge 1$, the homology of the product $S^{k-1}\times S^1$ can be computed by using the Künneth formula, yielding
    \[
    H_r(S^{k-1}\times S^1) \cong  \bigoplus_{i+j=r} H_i(S^{k-1}) \otimes H_j(S^1) .
    \]
    From this, the result follows using the well-known homology groups of the spheres.
\end{proof}

Now we are ready to compute the relative homology of the pair $(D^k\times S^1,S^{k-1}\times S^1)$. Note that here we are using singular homology, but we could also use \v{C}ech homology instead, since these pairs are triangulable and thus the two are isomorphic.

\begin{lemma}\label{lem:relative-homology-disk-sphere-circle}
    For any $k\ge 1$, we have
    \[
    H_r(D^k\times S^1,S^{k-1}\times S^1) \cong \begin{cases}
        \Fi, &\text{if } r=k,k+1, \\
        0, &\text{otherwise}.
    \end{cases}
    \]
\end{lemma}

\begin{proof}
    We use the long exact sequence in homology of the pair $(D^k\times S^1,S^{k-1}\times S^1)$. We start with the case $k\ge 3$. By \Cref{lem:homology-of-sphere}, the homology of $S^{k-1}\times S^1$ is non-zero only in degrees $0,1,k-1,k$, while $D^k\times S^1$ is homotopy equivalent to $S^1$ and thus has non-zero homology only in degrees $0,1$. In these non-zero degrees the homology is always isomorphic to $\Fi$. One can check that the embedding $S^{k-1}\times S^1 \hookrightarrow D^k \times S^1$ induces isomorphisms in homology in degrees $0$ and $1$. Thus the relevant parts of the long exact sequence are as follows:
    \begin{center}
    \begin{tikzcd}
        \cdots \ar[r] &0 \ar[r] &H_{k+1}(D^k\times S^1,S^{k-1}\times S^1) \ar[ddll] \\
        \\
        \underbrace{H_k(S^{k-1}\times S^1)}_{\cong \Fi} \ar[r] &0 \ar[r] &H_k(D^k\times S^1,S^{k-1}\times S^1) \ar[ddll] \\
        \\
        \underbrace{H_{k-1}(S^{k-1}\times S^1)}_{\cong \Fi} \ar[r] &0 \ar[r] &H_{k-1}(D^k\times S^1,S^{k-1}\times S^1) \ar[dl] \\
        &\text{\raisebox{-2pt}{\rotatebox{18}{$\cdots$}}} \ar[dl] \\
        \underbrace{H_1(S^{k-1}\times S^1)}_{\cong \Fi} \ar[r,"\cong"] &\underbrace{H_1(D^k\times S^1)}_{\cong \Fi} \ar[r] &H_1(D^k\times S^1,S^{k-1}\times S^1) \ar[ddll] \\
        \\
        \underbrace{H_0(S^{k-1}\times S^1)}_{\cong \Fi} \ar[r,"\cong"] &\underbrace{H_0(D^k\times S^1)}_{\cong \Fi} \ar[r] &H_0(D^k\times S^1,S^{k-1}\times S^1) \ar[ddll] \\
        \\
        \hspace{30pt}0  
    \end{tikzcd}
    \end{center}
    It now follows from exactness of the sequence that $H_r(D^k\times S^1,S^{k-1}\times S^1)=0$ for $r\neq k,k+1$ and also that the connecting homomorphism in degrees $k+1$ and $k$ is an isomorphism, hence completing the proof for the case $k\ge 3$.
    
    The cases $k=1$ and $k=2$ can be done analogously, only that the two non-zero parts of the long exact sequence overlap.
\end{proof}

The computations we just did are useful because of the following result, stating that there exists a relative homeomorphism between the pair $(D^k\times S^1,S^{k-1}\times S^1)$ and the unstable manifold of a closed orbit with its boundary collapsed to a point.

\begin{proposition}\label{prop:unstable-circle-relative-homeo}
    If $\beta$ is a closed orbit of index $k\ge 1$, then there exists a relative homeomorphism from $(D^k\times S^1,S^{k-1}\times S^1)$ to $(\overline{W^u(\beta)}/\partial W^u(\beta),*)$.
\end{proposition}

\begin{proof}
    Denote by $f\colon \R^k \times S^1 \to W^u(\beta)$ the diffeomorphism from \Cref{lem:unstable-manifold-homeomorphic} (the definition can be found in \cite[Section 2.2]{Smale1960MorseIineq}). Identifying the interior of $D^k$ with $\R^k$ via some fixed homeomorphism, we can view $f$ as a map from $\operatorname{Int}(D^k)\times S^1$ to $W^u(\beta)$. By \Cref{lem:cts-extension-rel-homeo}, we can extend this to a relative homeomorphism $\overline{f}\colon (D^k\times S^1,S^{k-1}\times S^1) \to (W^u(\beta)^\star,*)$. Composing this with the homeomorphism $W^u(\beta)^\star \cong \overline{W^u(\beta)}/\partial W^u(\beta)$ from \Cref{prop:quotient-compactification}, the result follows.
\end{proof}

Now we are ready to prove the analogous result to \Cref{cor:relative-cech-homology-unstable-fixpt} also for closed orbits.

\begin{corollary}\label{cor:relative-cech-homology-unstable-orbit}
    If $\beta$ is a closed orbit of index $k\ge 0$, then 
    \[
    \check{H}_r(\overline{W^u(\beta)},\partial W^u(\beta)) \cong \begin{cases}
        \Fi, &\text{if } $r=k+1, k$, \\
        0, &\text{otherwise}.
    \end{cases}
    \]
\end{corollary}

\begin{proof}
    If the index of $\beta$ is $k=0$, then $\overline{W^u(\beta)}=W^u(\beta)=\beta \cong S^1$ and $\partial W^u(\beta)=\emptyset$. In that case, the result is clear. For $k\ge 1$, we have 
    \begin{align*}        
    \check{H}_r(\overline{W^u(\beta)},\partial W^u(\beta)) 
    \cong \check{H}_r(\overline{W^u(\beta)}/\partial W^u(\beta),*) 
    &\cong \check{H}_r(D^k\times S^1,S^{k-1}\times S^1) \\
    &\cong H_r(D^k\times S^1,S^{k-1}\times S^1).
    \end{align*}
    The first isomorphism follows from \Cref{thm:rel-homeo-induce-iso} and \Cref{cor:quotient-rel-homeo}. 
    The second isomorphism follows from \Cref{thm:rel-homeo-induce-iso} and \Cref{prop:unstable-circle-relative-homeo}.
    The third isomorphism follows from \Cref{cor:cech-sing-iso}, since these are triangulable spaces.
    The result then follows from \Cref{lem:relative-homology-disk-sphere-circle}.
\end{proof}

\subsection{Canonical generators for the first page}\label{sec:canonical-generators}

Recall that $\Fix_k(v)$ is the set of fixed points of index $k$ and $\Orb_k(v)$ is the set of closed orbits of index $k$. Recall further that in \Cref{sec:filtration-induced-by-vector-field} we described a filtration $\emptyset = L_{-1} \subseteq L_0 \subseteq \cdots \subseteq L_n = M$ in terms of unstable manifolds of the singular elements of $v$. Inspired by the proof of \cite[Lemma 5.1]{Smale1960MorseIineq} and making use of the computations done in \Cref{sec:homology-computations}, we show that one can get canonical bases for the relative \v{C}ech homology groups of the pairs $(L_p,L_{p-1})$, where each fixed point contributes one basis element and each periodic orbit contributes two. This is useful for us since these will be the terms appearing on the first page of our spectral sequence.

Given $\beta \in \Fix(v)$, we denote by $\Fi \langle \beta \rangle$ the vector space of formal $\Fi$-linear combinations of $\beta$. If $\beta \in \Orb(v)$ is a closed orbit, we formally create two copies $\beta^+$ and $\beta^-$, and denote by $\Fi \langle \beta^+ \rangle$ and $\Fi \langle \beta^- \rangle$ the vector spaces of their formal $\Fi$-linear combinations, respectively.

\begin{proposition}\label{prop:rel-cech-basis}
    For every $p,r\ge 0$, we have 
    \[
    \check{H}_r(L_p,L_{p-1}) \cong 
    \bigoplus_{\substack{\beta \in \Fix_r(v)\\ \beta \in L_p\setminus L_{p-1}}} \Fi \langle \beta \rangle \oplus 
    \bigoplus_{\substack{\beta \in \Orb_{r-1}(v)\\ \beta \subseteq L_p\setminus L_{p-1}}} \Fi \langle \beta^+ \rangle \oplus
    \bigoplus_{\substack{\beta \in \Orb_r(v)\\ \beta \subseteq L_p\setminus L_{p-1}}} \Fi \langle \beta^- \rangle.
    \]
\end{proposition}

\begin{proof}
    For $p=0$ we have $L_0=\Fix_0(v) \sqcup \Orb_0(v)$. Thus,
    \begin{align*}
        \check{H}_r(L_0,L_{-1}) =\check{H}_r(L_0,\emptyset)  = \check{H}_r(\Fix_0(v) \sqcup \Orb_0(v)) 
        = \bigoplus_{\beta \in \Fix_0(v)} \check{H}_r(\{\beta\}) \oplus \bigoplus_{\beta \in \Orb_0(v)} \check{H}_r(\beta).
    \end{align*}
    Since $\beta\in \Fix_0(v)$ is a point and $\beta \in \Orb_0(v)$ is homeomorphic to a circle, the result in the case $p=0$ follows.
    For $p\ge 1$, we have
    \begin{align*}
        \check{H}_r(L_p,L_{p-1}) &\cong \check{H}_r(L_p/ L_{p-1},*) \\[1ex]
        &\cong \check{H}_r \left(\bigvee_{\substack{\beta \in \Fix(v)\\ \beta \in L_p\setminus L_{p-1}}} \overline{W^u}(\beta)/\partial W^u(\beta) \vee \bigvee_{\substack{\beta \in \Orb(v)\\ \beta \subseteq L_p\setminus L_{p-1}}} \overline{W^u}(\beta)/\partial W^u(\beta), * \right) \\[1em]
        &\cong \bigoplus_{\substack{\beta \in \Fix(v)\\ \beta \in L_p\setminus L_{p-1}}} \check{H}_r\left(\overline{W^u}(\beta)/\partial W^u(\beta),*\right) \oplus \bigoplus_{\substack{\beta \in \Orb(v)\\ \beta \subseteq L_p\setminus L_{p-1}}} \check{H}_r\left(\overline{W^u}(\beta)/\partial W^u(\beta),*\right) \\[1em]
        &\cong \bigoplus_{\substack{\beta \in \Fix_r(v)\\ \beta \in L_p\setminus L_{p-1}}} \Fi \langle \beta \rangle \oplus 
    \bigoplus_{\substack{\beta \in \Orb_{r-1}(v)\\ \beta \subseteq L_p\setminus L_{p-1}}} \Fi \langle \beta^+ \rangle \oplus
    \bigoplus_{\substack{\beta \in \Orb_r(v)\\ \beta \subseteq L_p\setminus L_{p-1}}} \Fi \langle \beta^- \rangle.
    \end{align*}
    The first isomorphism follows from \Cref{thm:rel-homeo-induce-iso} and \Cref{cor:quotient-rel-homeo}. The second one follows from the definition of the $L_p$, since $L_p \setminus L_{p-1}$ is the disjoint union of all those unstable manifolds that are not contained in $L_{p-1}$ but whose boundary is. The third one follows from \Cref{prop:homology-iso-wedge-directsum}. The fourth isomorphism follows from \Cref{cor:relative-cech-homology-unstable-fixpt,cor:relative-cech-homology-unstable-orbit}.
\end{proof}

\section{The chain complex of a Morse-Smale vector field}\label{sec:putting-together}

Here we connect all the various things from the previous sections and subsections in order to conclude our goal of associating algebraic invariants with Morse-Smale vector fields.

Let $v$ be a Morse-Smale vector field on a closed smooth manifold $M$. We want to assign a chain complex $C_\bullet(v)$ to $v$ according to the following pipeline:

\tikzstyle{process2} = [rectangle, 
minimum width=2.2cm, 
minimum height=1cm, 
text centered,
text width=2.2cm, 
draw=black, 
fill=orange!30]

\tikzstyle{process3} = [rectangle, 
minimum width=3.6cm, 
minimum height=1cm, 
text centered,
text width=3.6cm, 
draw=black, 
fill=orange!30]

\begin{center}
	\begin{tikzpicture}[node distance=4.8cm]
		
		\node (vec) [process2] {\footnotesize Morse-Smale vector field ${v \in \X_{MS}(M)}$};
		
		\node (filt) [process3, right of=vec] {\footnotesize {filtration by unstable manifolds}\\ $ L_0 \subseteq L_1 \subseteq \cdots \subseteq L_n = M$};
		
		\node (SS) [process3, right of=filt] {\footnotesize spectral sequence in \v{C}ech homology ${E^1_{p,q} = \check{H}_{p+q}(L_p,L_{p-1})}$};
		
		\node (CC) [process2, right of=SS] {\footnotesize {chain complex}\linebreak $C_\bullet(E(v))$};
		
		\draw[->] (vec)--(filt);
		\draw[->] (filt) -- (SS);
		\draw[->] (SS)--(CC);
	\end{tikzpicture}
\end{center}

\textbf{1st arrow:} Given a Morse-Smale vector field $v$, we consider the filtration of the underlying manifold by unstable manifolds of fixed points and closed orbits described in \Cref{sec:filtration-induced-by-vector-field}. This is defined inductively: In each step we add in those unstable manifolds whose boundaries attach only to things that have been added before. 
Explicitly, $L_{-1} := \emptyset$ and for $p\ge0$, if $L_{p-1}$ has been defined, then 
\[
L_p := \displaystyle \bigcup_{\substack{\beta \in \operatorname{Crit}(v), \\ \partial W^u(\beta) \subseteq L_{p-1}}} W^u(\beta).
\]

\textbf{2nd arrow:} From this filtration we consider the spectral sequence $E(v)$  whose first page consists of the relative homologies of adjacent steps in the filtration, which exists according to \Cref{prop:spectral-sequence-from-filtration}. The differentials are defined in terms of the connecting homomorphism from the long exact sequence in homology.\\

\textbf{3rd arrow:} We follow \Cref{sec:chain-complex-of-spectral-sequence} to construct a chain complex $ C_\bullet(v):= C_\bullet(E(v))$ from this spectral sequence. According to \Cref{prop:chain-complex-of-spectral-sequence}, the vector spaces of this chain complex are isomorphic to the direct sums along the diagonals on the first page, which can be determined through \Cref{prop:rel-cech-basis}. \\ 

This pipeline is canonical in the sense that no arbitrary choices are made, such as auxiliary functions or filtrations. In the next subsections, we explore the main properties of the resulting chain complex together with some applications and explicit examples.

\subsection{Properties of $C_\bullet(v)$}

The above pipeline allows us to  construct a chain complex from a Morse-Smale vector field, free of arbitrary choices. Moreover, the homology of this chain complex agrees with the homology of the underlying manifold, and the vector spaces in the chain complex are generated by the singular elements of the vector field. We summarize this in the following result.

\begin{theorem}\label{thm:conclusion}
    If $v$ is a Morse-Smale vector field on a closed smooth manifold $M$ of dimension $m$, then there exists a chain complex $C_\bullet=C_\bullet(v)$ with the following properties:
    \begin{enumerate}[(i)]
        \item $H_*(C_\bullet) \cong H_*(M)$,
        \item $\dim(C_k) = |\Fix_k(v)|+|\Orb_{k-1}(v)|+|\Orb_k(v)|$ for $k=0,1,\ldots,m$.
    \end{enumerate}
\end{theorem}

\begin{proof}
    By \Cref{prop:exist-number}, the unstable manifold filtration of $v$ eventually contains all of $M$, i.e. $\emptyset \subseteq L_0 \subseteq L_1 \subseteq \cdots \subseteq L_n =M$. By \Cref{prop:spectral-sequence-from-filtration}, the spectral sequence $E=E(v)$ constructed from this filtration converges to $\check{H}_*(M)$, which equals $H_*(M)$, since $M$ is a manifold and thus triangulable. By \Cref{prop:chain-complex-from-spectral-sequence-homology}, it follows that $H_*(C_\bullet) \cong H_*(M)$. 

    It remains to show the formula for the dimensions. Let us introduce additional notation to further decompose the sets of fixed points and closed orbits, according to the time at which they enter the filtration of $M$ by the sets $L_p$. More precisely, for $r=0,1,\ldots,m$, we have $\Fix_r(v)=\Fix_r^0(v) \sqcup \Fix_r^1(v) \sqcup \ldots \sqcup \Fix_r^n(v)$, and for $r=0,1,\ldots,m-1$, we have $\Orb_r(v) = \Orb_r^0(v) \sqcup \Orb_r^1(v) \sqcup \ldots \sqcup \Orb_r^n(v)$, where, for $p=0,1,\ldots,n$,
    \begin{align*}
        \Fix_r^p(v) &:= \{ \beta \in \Fix_r(v) \mid \beta \in L_p \setminus L_{p-1}\}, \\
        \Orb_r^p(v) &:= \{ \beta \in \Orb_r(v) \mid \beta \subseteq L_p \setminus L_{p-1}\}.
    \end{align*}
    We can now compute the dimension of $C_k$, namely
    \begin{align*}
        \dim(C_k) 
        = \sum_{p+q=k} \dim(E^1_{p,q}) 
        &= \sum_{p+q=k} \dim(\check{H}_k(L_p,L_{p-1})) \\
        &= \sum_{p+q=k} |\Fix_k^p(v)|+|\Orb_{k-1}^p(v)|+|\Orb_k^p(v)| \\
        &= |\Fix_k(v)|+|\Orb_{k-1}(v)|+|\Orb_k(v)|,
    \end{align*}
    where the first equality follows from \Cref{prop:chain-complex-of-spectral-sequence} and \Cref{eq:first-page-filtration-iso}, the second by \Cref{prop:spectral-sequence-from-filtration}, the third from \Cref{prop:rel-cech-basis}, and the fourth from realizing that the sum over $p+q=k$ is equivalent to simply a sum over all $p$ and then reordering the terms.
\end{proof}

We now prove that the isomorphism type of the chain complex assigned to a Morse-Smale vector field is invariant under topological equivalence.

\begin{proposition}
    If $v$ and $w$ are topologically equivalent Morse-Smale vector fields on a closed smooth manifold $M$, then the chain complexes $C_\bullet(v)$ and $C_\bullet(w)$ are isomorphic.
\end{proposition}

\begin{proof}
    Denote by $h\colon M \to M$ the topological equivalence between $v$ and $w$. This means that $h$ is a homeomorphism that maps fixed points of $v$ to fixed points of $w$ and closed orbits of $v$ to closed orbits of $w$, preserving the indices. Moreover, $h$ restricts to homeomorphisms between the unstable manifolds of $v$ and $w$. Thus we have that corresponding fixed points get added at the same time in the filtrations corresponding to $v$ and $w$. Writing $\emptyset \subseteq L_0 \subseteq L_1 \subseteq \cdots \subseteq L_n = M$ for the filtration associated with $v$ and $\emptyset \subseteq L_0' \subseteq L_1' \subseteq \cdots \subseteq L_n' = M$ for the filtration associated with $w$, as described in \Cref{sec:filtration-induced-by-vector-field}, we get that $L_p \cong L_p'$ for all $p=0,1,\ldots,n$. It follows that we have an isomorphism of spectral sequences $E(v) \cong E(w)$. Since the construction of $C_\bullet(v)$ from $E(v)$ and of $C_\bullet(w)$ from $E(w)$ is purely algebraic and follows the same step, this induces an isomorphism of chain complexes $C_\bullet(v) \cong C_\bullet(w)$. 
\end{proof}

\subsection{The gradient-like case}

For a gradient-like Morse-Smale vector field $v$, a chain complex with the desired properties such as the one constructed above already exists in the form of the {\em Morse complex} $\MC_\bullet(v)$. The vector spaces $\MC_k(v)$ are generated by the fixed points of index $k$ of $v$ and the boundary operator is defined by counting flow lines between fixed points of adjacent indices, see for example \cite{BanyagaLecturesOnMorseHomology} for more details. We now show that for a gradient-like Morse-Smale vector field $v$, our chain complex $C_\bullet(v)$ is isomorphic to the Morse complex $\MC_\bullet(v)$ of $v$. 

\begin{corollary}
    If $v$ is a gradient-like Morse-Smale vector field on $M$, then $C_\bullet(v) \cong \MC_\bullet(v)$.
\end{corollary}

\begin{proof}
    Since every chain complex of vector spaces can be uniquely decomposed into homology summands whose differentials are zero and contractible summands with null homology (see e.g. Exercise 1.1.3 in \cite{weibel1994introduction}), it suffices to show that $\dim(C_k)=\dim(\MC_k)$ and $\dim (H_k(C_\bullet)) = \dim (H_k(\MC_\bullet))$ for all $k$. For the former, we have that $\dim(\MC_k))$ equals the number of fixed points of index $k$ of $v$. The same holds for $\dim(C_k)$ by \Cref{thm:conclusion}. As for the homology, again by \Cref{thm:conclusion} (and the Morse homology theorem), both of these are equal to the singular homology of $M$. 
\end{proof}

\subsection{The Morse inequalities}\label{sec:morse-inequalities}

We now want to use the chain complex $C_\bullet(v)$ assigned to a Morse-Smale vector field $v$ on a closed smooth manifold $M$ in order to derive the strong Morse inequalities from \cite{Smale1960MorseIineq}. Let us first recall the algebraic Morse inequalities, namely for any compact chain complex $C_\bullet \in \Ch$, and any $q\ge 0$, we have
\begin{equation}\label{eq:algebraic-morse-inequalities}
    \sum_{i=0}^q(-1)^{q+i} \dim(C_i) \ge \sum_{i=0}^q(-1)^{q+i} \dim(H_i(C_\bullet)).
\end{equation}
These inequalities can be shown by standard linear algebra arguments, doing an induction over $q$. Moreover, if $q$ is the maximal value for which $C_q\neq 0$, then the inequality from (\ref{eq:algebraic-morse-inequalities}) becomes an equality. The value on both sides of the equation in that case is the Euler characteristic. Given these algebraic Morse inequalities, we can derive the Morse inequalities from \cite[Theorem 1.1]{Smale1960MorseIineq}, using the mere existence of a chain complex as in \Cref{thm:conclusion}.

\begin{corollary}\label{cor:morse-ineq}
    If $v$ is a Morse-Smale vector field on a closed smooth manifold $M$, for $q\ge0$, write $R_q:=\dim(H_q(M))$ and $M_q:=|\Fix_q(v)|+|\Orb_{q-1}(v)|+|\Orb_q(v)|$. Then, the following inequalities hold:
    \[
    \sum_{i=0}^q(-1)^{q+i}M_i \ge \sum_{i=0}^q(-1)^{q+i}R_i.
    \]
\end{corollary}

\begin{proof}
    From \Cref{thm:conclusion} and \Cref{eq:algebraic-morse-inequalities} it follows that
    \[
    \sum_{i=0}^q(-1)^{q+i}M_i = \sum_{i=0}^q(-1)^{q+i}\dim(C_i(v)) \ge \sum_{i=0}^q(-1)^{q+i}\dim(H_i(C_\bullet(v))) =  \sum_{i=0}^q(-1)^{q+i}R_i. \qedhere
    \]
\end{proof}

\subsection{Examples}

By \Cref{prop:rel-cech-basis}, we get can endow the terms on the first page of $E$ with bases given by the fixed points and closed orbits of $v$. In particular, this implies many zero terms, as we show in the following proposition.

\begin{proposition}\label{prop:zero-terms}
    Let $\emptyset = L_{-1} \subseteq L_0 \subseteq L_1 \subseteq \cdots \subseteq L_n = M$ be the unstable manifolds filtration defined with respect to a Morse-Smale vector field $v$ as described in \Cref{sec:filtration-induced-by-vector-field}. Then, the following statements hold.
    \begin{enumerate}[(i)]
        \item $\check{H}_i(L_0,\emptyset)=0$ for $i\ge 2$.\label{item:zero-terms-1}
        \item $\check{H}_0(L_p,L_{p-1})=0$ for $p\ge 1$.\label{item:zero-terms-2}
        \item If $\dim(M)=2$, then $n=2$ and $\check{H}_i(L_1,L_0)=\check{H}_i(L_2,L_1)=0$ for $i\ge 3$.\label{item:zero-terms-3}
    \end{enumerate}
\end{proposition}

\begin{proof}
    If $\beta$ is a fixed point or closed orbit of index zero, then $\partial W^u(\beta)=\emptyset$ and thus $W^u(\beta) \subseteq L_0$.
    By \Cref{lem:unstable-manifold-homeomorphic}, the only fixed points and closed orbits whose unstable manifolds are closed subsets of $M$ are those of index zero. Therefore, any fixed point or closed orbit of larger index has at least one other critical element in its boundary and thus cannot be contained in $L_0$. In other words, $L_0$ consists precisely of all fixed points and closed orbits of index 0. 
    So, it is a disjoint union of finitely many points and circles; in particular, it is triangulable. So, its \v{C}ech homology agrees with singular homology by \Cref{cor:cech-sing-iso} and is potentially non-zero only in degrees zero and one, which proves $(\ref{item:zero-terms-1})$. 

    Note that by \Cref{prop:rel-cech-basis}, $\check{H}_0(L_p,L_{p-1})$ is generated by fixed points of index zero and closed orbits of index zero that are present in $L_p$ but not in $L_{p-1}$. As we have just proved that none of them gets added later than $L_0$, $(\ref{item:zero-terms-2})$ follows.

    On a two-dimensional manifold, the index of a fixed point can be at most $2$, and the index of a closed orbit can be at most $1$, i.e. $\Fix_r(v)=\emptyset$ for $r\ge 3$, and $\Orb_r(v)=\emptyset$ for $r\ge 2$. Thus, it follows again from \Cref{prop:rel-cech-basis} that $\check{H}_i(L_1,L_0)=\check{H}_i(L_2,L_1)=0$ for $i\ge 3$. This proves $(\ref{item:zero-terms-3})$.
\end{proof}

As a consequence, for a Morse-Smale vector field on a surface the first page of the spectral sequence from \Cref{sec:filtration-induced-by-vector-field}, with $E^1_{p,q} = \check{H}_{p+q}(L_p,L_{p-1})$, looks as follows (where we do not write the zero terms):

\begin{center}
\begin{tabular}{c|c}
    \renewcommand{\arraystretch}{3}
    \begin{tabular}{c}
    $q=1$ \\
    $q=0$ \\
    $q=-1$
    \end{tabular}
    & 
    \begin{tikzcd}
        \check{H}_1(L_0,\emptyset)  &\check{H}_2(L_1,L_0) \ar[l] \\
        \check{H}_0(L_0,\emptyset)  &\check{H}_1(L_1,L_0) \ar[l]  &\check{H}_2(L_2,L_1) \ar[l] \\
        &&\check{H}_1(L_2,L_1)
    \end{tikzcd} \\ \hline
    &   
    \setlength{\tabcolsep}{27pt}
    \begin{tabular}{ccc}
        $p=0$ & $p=1$ & $p=2$
    \end{tabular}
\end{tabular}
\end{center}
The chain complex $C_\bullet(v) := C_\bullet(E(v))$ assigned to this spectral sequence according to \Cref{sec:chain-complex-of-spectral-sequence}, in the 2D case simplifies to
\begin{center}
    \begin{tikzcd}[ampersand replacement=\&,row sep = 2.4cm]
        C_2(v) = E^1_{1,1} \oplus \ker(d^1_{2,0}) \oplus E^1_{2,0}/\ker(d^1_{2,0}) \ar[d,"{\begin{bmatrix}
            0 &d^2_{2,0} &0 \\
            d^1_{1,1} &0 &0 \\
            0 &0 & \overline{d}^1_{2,0} \\
            0 &0 &0 
        \end{bmatrix}}"] \\
        C_1(v) = E^1_{0,1}/\im(d^1_{1,1}) \oplus \im(d^1_{1,1}) \oplus E^1_{1,0} \oplus E^1_{2,-1} \ar[d,"{\begin{bmatrix}
            0 &0 &0 &d^2_{2,-1} \\
            0 &0 &d^1_{1,0} &0 
        \end{bmatrix}}"] \\
        C_0(v) = E^1_{0,0}/\im(d^1_{1,0}) \oplus \im(d^1_{1,0}).
    \end{tikzcd}
\end{center}

We now give two examples of Morse-Smale vector fields $v$ on $S^2$ for which we compute the chain complex $C_\bullet(v)$. To simplify the expressions, we assume that $\Fi=\Z/2\Z$.

\begin{figure}[t]
    \centering
    \includegraphics[width=8cm]{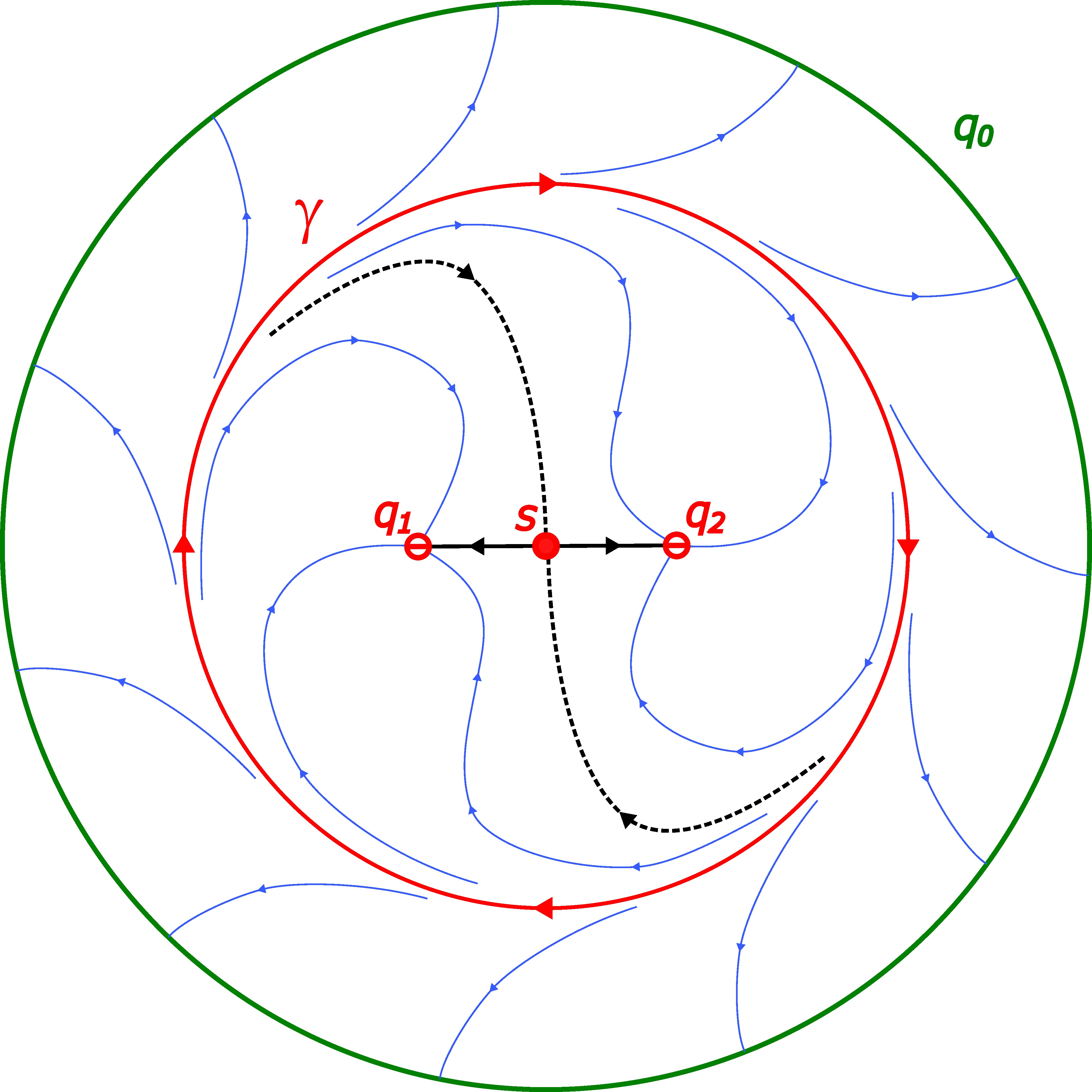}
    \caption{Morse-Smale vector field on $S^2$ with three sinks, one saddle, and one repelling orbit}
    \label{fig:vector-field-1}
\end{figure}

\begin{example}
    We draw the 2-sphere as a disk, where the boundary represents a single point $q_0$. Consider the vector field $v \in \X_{MS}(S^2)$ displayed in \Cref{fig:vector-field-1}. It has three sinks $q_0$, $q_1$, $q_2$, one saddle $s$ and one closed orbit $\gamma$ of index $1$. The filtration $\emptyset \subseteq L_0 \subseteq L_1 \subseteq L_2 = S^2$ is the following: $L_0 = \{q_0,q_1,q_2\}$, $L_1\setminus L_0 = W^u(s)$, $L_2\setminus L_1 = W^u(\gamma)$. In this example, the spaces $L_p$ are all triangulable, so their \v{C}ech homology agrees with singular homology. According to \Cref{prop:spectral-sequence-from-filtration}, there is a spectral sequence $E$ converging to $H_*(M)$ with the first page given by $E^1_{p,q} = H_{p+q}(L_p,L_{p-1})$. We endow these terms with canonical bases according to \Cref{prop:rel-cech-basis}. The first page of the spectral sequence thus looks as follows:
    \begin{center}
    \begin{tabular}{c|c}
    \renewcommand{\arraystretch}{4.5}
    \begin{tabular}{c}
    $q=1$ \\
    $q=0$ \\
    $q=-1$\vspace{10pt}
    \end{tabular}
    & 
    \begin{tikzcd}
        \underbrace{H_1(L_0,\emptyset)}_{=0}  &\underbrace{H_2(L_1,L_0)}_{=0} \ar[l] \\
        \underbrace{H_0(L_0,\emptyset)}_{= \Fi\langle q_0 \rangle \oplus \Fi\langle q_1 \rangle \oplus \Fi\langle q_2\rangle}  &\underbrace{H_1(L_1,L_0)}_{= \Fi\langle s \rangle} \ar[l,"d^1_{1,0}"{above}]  &\underbrace{H_2(L_2,L_1)}_{= \Fi\langle \gamma^+ \rangle} \ar[l,"d^1_{2,0}"{above}] \\
        &&\underbrace{H_1(L_2,L_1)}_{= \Fi\langle \gamma^- \rangle},
    \end{tikzcd} \\ \hline
    &   
    \setlength{\tabcolsep}{33pt}
    \begin{tabular}{ccc}
        \hspace{-2pt}$p=0$ & $p=1$ & $p=2$
    \end{tabular}
\end{tabular}
\end{center}
    where the differentials with respect to these bases are given by $d^1_{1,0} = \begin{bmatrix} 0\\ 1\\ 1 \end{bmatrix}$ and $d^1_{2,0} = \begin{bmatrix} 0 \end{bmatrix}$.
    The second page looks as follows:
\begin{center}
\begin{tabular}{c|c}
    \renewcommand{\arraystretch}{3.2}
    \begin{tabular}{c}
    $q=1$ \\
    $q=0$ \\
    $q=-1$\vspace{0pt}
    \end{tabular}
    &
    \begin{tikzcd}
        \hspace{20pt}0\hspace{20pt}  &0 \\
        E^2_{0,0} = E^1_{0,0}/\im(d^1_{1,0})  &\overbrace{\ker(d^1_{1,0})}^{= 0}  &E^2_{2,0}=E^1_{2,0} = \Fi\langle \gamma^+\rangle \\
        &&E^2_{2,-1}=E^1_{2,-1} = \Fi\langle \gamma^-\rangle. \ar[ull,"d^2_{2,-1}"]
    \end{tikzcd} \\ \hline
    &   
    \setlength{\tabcolsep}{35pt}
    \begin{tabular}{ccc}
        \hspace{-2pt}$p=0$ & $p=1$ & $p=2$
    \end{tabular}
\end{tabular}
\end{center}
    Note that $E^2_{0,0} \cong \langle q_0,q_1,q_2 \mid q_1=q_2\rangle$. We assign to $E^2_{0,0}$ the basis $\{[q_0],[q_1]\}$. The resulting matrix for $d^2_{2,-1}$ is $\begin{bmatrix} 1\\ 1 \end{bmatrix}$. To $\im(d^1_{1,0})$ we assign the basis $\{q_1+q_2\}$. In conclusion, this yields the following based chain complex:
    \begin{center}
    \begin{tikzcd}[ampersand replacement=\&,row sep = 2.0cm]
        C_2(v) = \ker(d^1_{2,0}) = \langle \gamma^+\rangle \ar[d,"{\begin{bmatrix}
            0 \\
            0
        \end{bmatrix}}"] \\
        C_1(v) = E^1_{1,0} \oplus E^1_{2,-1} = \langle s,\gamma^-\rangle \ar[d,"{\begin{bmatrix}
            0 &d^2_{2,-1} \\
            d^1_{1,0} &0 
        \end{bmatrix}=\begin{bmatrix}
            0 &1 \\
            0 &1 \\
            1 &0  
        \end{bmatrix}}"] \\
        C_0(v) = E^1_{0,0}/\im(d^1_{1,0}) \oplus \im(d^1_{1,0}) = \langle [q_0],[q_1],q_1+q_2\rangle.
    \end{tikzcd}
    \end{center}
\end{example}

\begin{example}
    Consider the vector field $v \in \X_{MS}(S^2)$ displayed in \Cref{fig:vector-field-2}. It has four sinks $q_0$, $q_1$, $q_2$, $q_3$, four saddles $s_1$, $s_2$, $s_3$, $s_4$, two sources $p_1$, $p_2$, and one closed orbit $\gamma$ of index $1$. The filtration $\emptyset = \subseteq L_0 \subseteq L_1 \subseteq L_2 = S^2$ is the following: 
    \begin{align*}
        L_0 &= \{q_0,q_1,q_2,q_3\}, &L_1\setminus L_0 &= W^u(s_1) \sqcup W^u(s_2) \sqcup W^u(s_3) \sqcup W^u(s_4),\\
        &&L_2\setminus L_1 &= W^u(p_1) \sqcup W^u(p_2) \sqcup W^u(\gamma).
    \end{align*}
    As in the previous example, these are all triangulable, so we may consider singular homology instead of \v{C}ech homology. Again by \Cref{prop:spectral-sequence-from-filtration,prop:rel-cech-basis} we get a spectral sequence $E(v)$ with the first page looking as follows:
    \begin{center}
    \begin{tabular}{c|c}
    \renewcommand{\arraystretch}{4.2}
    \begin{tabular}{c}
    $q=1$ \\
    $q=0$ \\
    $q=-1$\vspace{10pt}
    \end{tabular}
    & 
    \begin{tikzcd}
        \underbrace{H_1(L_0,\emptyset)}_{=0}  &\underbrace{H_2(L_1,L_0)}_{=0} \ar[l] \\
        \underbrace{H_0(L_0,\emptyset)}_{= \bigoplus_{i=0}^3 \Fi\langle q_i \rangle}  &\underbrace{H_1(L_1,L_0)}_{= \bigoplus_{i=1}^4 \Fi\langle s_i \rangle} \ar[l,"d^1_{1,0}"{above}]  &\underbrace{H_2(L_2,L_1)}_{= \Fi\langle p_1 \rangle \oplus \Fi\langle p_2 \rangle \oplus \Fi\langle \gamma^+ \rangle} \ar[l,"d^1_{2,0}"{above}] \\
        &&\underbrace{H_1(L_2,L_1)}_{= \Fi\langle \gamma^- \rangle}.
    \end{tikzcd} \\ \hline
    &   
\setlength{\tabcolsep}{30pt}
    \begin{tabular}{ccc}
        \hspace{-25pt}$p=0$ & $p=1$ & $p=2$
    \end{tabular}
\end{tabular}
\end{center}
    The differentials with respect to these bases are given by 
    \[
    d^1_{1,0} = 
    \begin{bmatrix} 
        0&0&0&0 \\ 
        1&0&1&0 \\ 
        1&1&0&1 \\
        0&1&1&1
    \end{bmatrix} 
    \qquad \text{and} \qquad 
    d^1_{2,0} = 
    \begin{bmatrix} 
        1&0&1 \\
        0&1&1 \\
        1&0&1 \\
        1&1&0
    \end{bmatrix}.
    \]
    In order to explicitly write down the resulting chain complex, we need to determine bases for the subspaces $\im(d^1_{1,0}) \subseteq E^1_{0,0}$, $\ker(d^1_{2,0}) \subseteq E^1_{2,0}$ and for the quotient spaces $E^1_{0,0} / \im(d^1_{1,0})$, $E^1_{2,0} / \ker(d^1_{2,0})$. 

\begin{figure}[t]
    \centering
    \includegraphics[width=8cm]{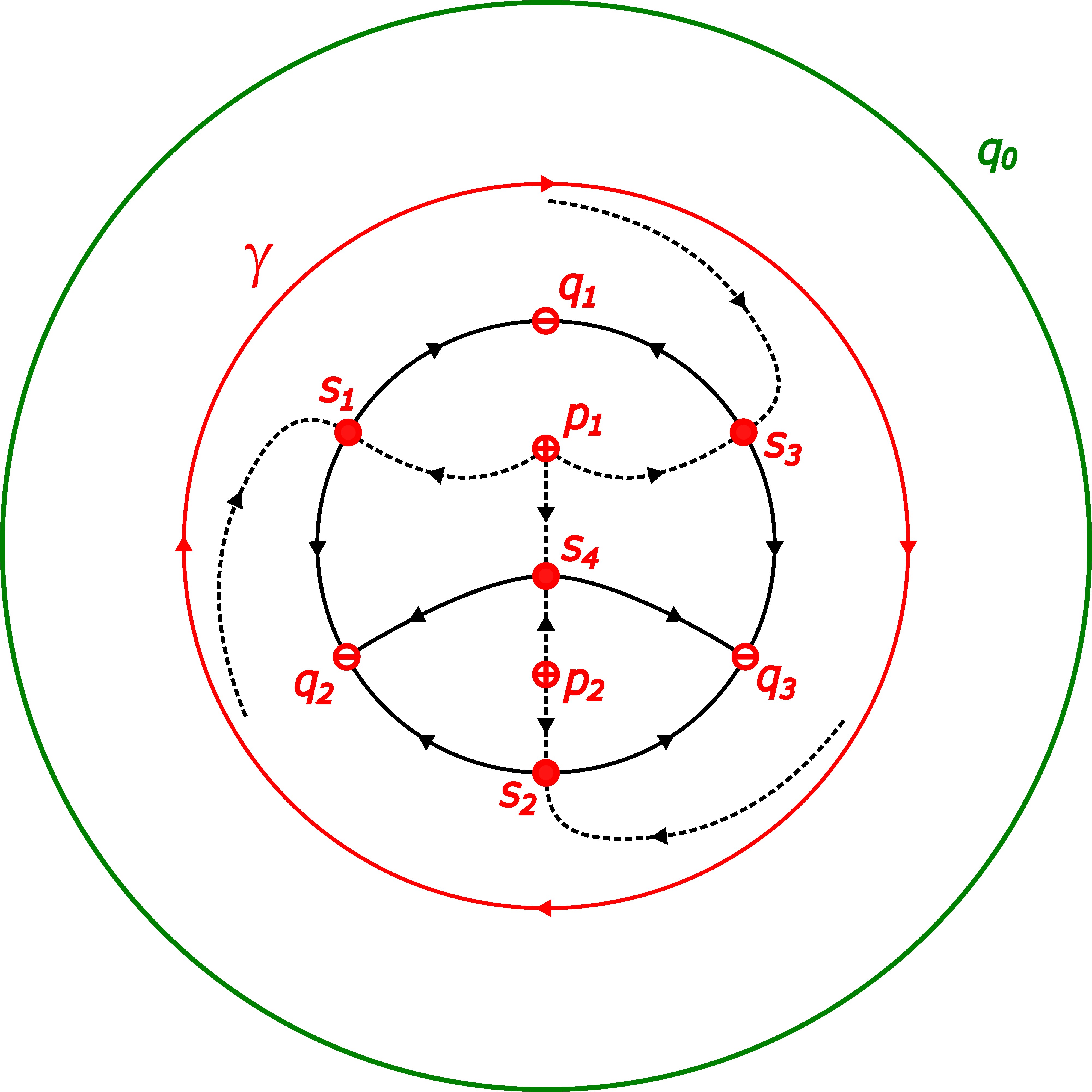}
    \caption{Morse-Smale vector field on $S^2$ with four sinks, four saddles, two sources, and one repelling orbit}
    \label{fig:vector-field-2}
\end{figure}

    \paragraph{Basis for subspaces:} We make use of the reduced row echelon form to endow subspaces of a given vector space with canonical bases. A more detailed description of the row echelon form and algorithms for computing it can be found in many linear algebra textbooks, see e.g. \cite{LayLinAlgAppl2014}. We take an arbitrary generating set, write their representation vectors as rows into a matrix, and bring the matrix to reduced echelon form. The non-zero rows of the resulting matrix will be the basis. In the case of $\im(d^1_{1,0}) \subseteq E^1_{0,0}$, we can start with the columns of $d^1_{1,0}$. This yields the following row echelon form:
    \[
    \begin{bmatrix}
        0 &1 &1 &0 \\
        0 &0 &1 &1 \\
        0 &1 &0 &1
    \end{bmatrix}
    \rightsquigarrow
    \begin{bmatrix}
        0 &1 &1 &0 \\
        0 &0 &1 &1 \\
        0 &0 &1 &1
    \end{bmatrix}
    \rightsquigarrow
    \begin{bmatrix}
        0 &1 &1 &0 \\
        0 &0 &1 &1 \\
        0 &0 &0 &0
    \end{bmatrix}
    \rightsquigarrow
    \begin{bmatrix}
        0 &1 &0 &1 \\
        0 &0 &1 &1 \\
        0 &0 &0 &0
    \end{bmatrix}.
    \]
    The resulting basis for $\im(d^1_{1,0})$ is thus $\{q_1+q_3,q_2+q_3\}$.
    For $\ker(d^1_{2,0}) \subseteq E^1_{2,0}$, a basis is given by $\{p_1+p_2+\gamma^+\}$. In this case, there is nothing more to do, since the basis consists only of one element.

    \paragraph{Basis for quotient spaces:} We choose the basis $\{[q_0],[q_1]\}$ for $E^2_{0,0} = E^1_{0,0} / \im(d^1_{1,0})$ and the basis $\{[p_1],[p_2]\}$ for $E^2_{2,0} = E^1_{2,0} / \ker(d^1_{2,0})$. The other terms of the second page are either zero or agree with the first page. We can thus write the second page:
    \begin{center}
    \begin{tabular}{c|c}
    \renewcommand{\arraystretch}{3.5}
    \begin{tabular}{c}
    $q=1$ \\
    $q=0$ \\
    $q=-1$\vspace{0pt}
    \end{tabular}
    & 
    \begin{tikzcd}
        \hspace{20pt}0\hspace{20pt}  &0 \\
        E^2_{0,0} = \Fi\langle [q_0],[q_1] \rangle  &\overbrace{\ker(d^1_{1,0})}^{= 0}  &E^2_{2,0}= \Fi\langle p_1+p_2+\gamma^+\rangle \\
        &&E^2_{2,-1}=E^1_{2,-1} = \Fi\langle \gamma^-\rangle. \ar[ull,"d^2_{2,-1}"]
    \end{tikzcd} \\ \hline
    &   
\setlength{\tabcolsep}{30pt}
    \begin{tabular}{ccc}
        \hspace{-25pt}$p=0$ & $p=1$ & $p=2$
    \end{tabular}
\end{tabular}
\end{center}
    The matrix representation of the non-zero differential with respect to these bases is 
    $d^2_{2,-1}~=~\begin{bmatrix}
        1\\
        1
    \end{bmatrix}$.
    Hence, we get the following chain complex:
    \begin{center}
    \begin{tikzcd}[ampersand replacement=\&,row sep = 2.8cm]
        C_2(v) = \ker(d^1_{2,0}) \oplus E^1_{2,0}/\ker(d^1_{2,0}) = \Fi\langle p_1+p_2+\gamma^+ \rangle \oplus \Fi\langle [p_1] \rangle \oplus \Fi\langle [p_2]\rangle \ar[d,"{\begin{bmatrix}
            0 &\overline{d}^1_{2,0} \\
            0 &0 
        \end{bmatrix}=\begin{bmatrix}
            0 &1 &0 \\
            0 &0 &1 \\
            0 &1 &0 \\
            0 &1 &1 \\
            0 &0 &0
        \end{bmatrix}}"] \\
        C_1(v) = E^1_{1,0} \oplus E^1_{2,-1} = \Fi\langle s_1 \rangle \oplus \Fi\langle s_2 \rangle \oplus \Fi\langle s_3 \rangle \oplus \Fi\langle s_4 \rangle \oplus \Fi\langle \gamma^-\rangle \ar[d,"{\begin{bmatrix}
            0 &d^2_{2,-1} \\
            d^1_{1,0} &0 
        \end{bmatrix}=\begin{bmatrix}
            0 &0 &0 &0 &1 \\
            0 &0 &0 &0 &1 \\
            1 &0 &1 &0 &0 \\
            1 &1 &0 &1 &0
        \end{bmatrix}}"] \\
        C_0(v) = E^1_{0,0}/\im(d^1_{1,0}) \oplus \im(d^1_{1,0}) = \Fi\langle [q_0] \rangle \oplus \Fi\langle [q_1] \rangle \oplus \Fi\langle q_1+q_3 \rangle \oplus \Fi\langle q_2+q_3\rangle.
    \end{tikzcd}
    \end{center}
\end{example}

\bibliographystyle{amsplain}
\bibliography{Bibliography}

\end{document}

%% file: preamble.tex
\usepackage{graphicx}
\usepackage{xcolor}
\usepackage{hyperref}
\usepackage{dsfont}

\usepackage{tikz-cd} %For nice commutative diagrams
\usepackage{enumerate} %For using (i), (ii) etc in enumerate

\usepackage{caption}
\usepackage{subcaption}

\usepackage[margin=1in, top=0.5in]{geometry}

\usepackage{amsmath,amsfonts,amssymb,amsthm,mathrsfs}

\usepackage{cleveref}

\theoremstyle{plain} %bold title, italic body text

\newtheorem{theorem}{Theorem}[section]
\newtheorem{lemma}[theorem]{Lemma}
\newtheorem{proposition}[theorem]{Proposition}
\newtheorem{corollary}[theorem]{Corollary}

\theoremstyle{definition} % body text not italic

\newtheorem{remark}[theorem]{Remark}
\newtheorem{definition}[theorem]{Definition}

\newtheorem{example}[theorem]{Example}

\newcommand{\R}{\mathbb{R}}
\newcommand{\Z}{\mathbb{Z}}

\newcommand{\Fi}{\mathbb{F}}
\newcommand{\E}{\mathcal{E}}
\newcommand{\Oo}{\mathcal{O}}
\newcommand{\X}{\mathfrak{X}}
\newcommand{\eps}{\varepsilon}

\DeclareMathOperator{\im}{im} %image
\DeclareMathOperator{\Ch}{Ch} %category of chain complexes
\DeclareMathOperator{\MC}{MC} %the Morse complex
\DeclareMathOperator{\ind}{ind} %index
\DeclareMathOperator{\Fix}{Fix} %fixed points
\DeclareMathOperator{\Orb}{Orb} %closed orbits